\documentclass[11 pt, reqno]{amsart}

\usepackage[latin1]{inputenc}
\usepackage{color, amsmath,amsthm,amssymb,amscd}
\usepackage{bbm}

\newtheorem{thm}{Theorem}[section]
 \newtheorem{cor}[thm]{Corollary}
 
 \newtheorem{lem}[thm]{Lemma}
 \newtheorem{prop}[thm]{Proposition}

 \theoremstyle{definition}
 \newtheorem{df}[thm]{Definition}
 \theoremstyle{remark}
 \newtheorem{rem}[thm]{Remark}
 
 \numberwithin{equation}{section}

\def\be#1 {\begin{equation} \label{#1}}
\newcommand{\ee}{\end{equation}}

\def\sqw{\hbox{\rlap{\leavevmode\raise.3ex\hbox{$\sqcap$}}$%
\sqcup$}}
\def\findem{\ifmmode\sqw\else{\ifhmode\unskip\fi\nobreak\hfil
\penalty50\hskip1em\null\nobreak\hfil\sqw
\parfillskip=0pt\finalhyphendemerits=0\endgraf}\fi}

\title[Bounds of Sobolev norms]{Growth of higher Sobolev norms for energy critical NLS on irrational tori: Small energy case}

\author{Yu Deng}

\begin{document}

\maketitle

\begin{abstract} We consider the energy critical nonlinear Schr\"{o}dinger equation on generic irrational tori $\mathbb{T}_{\lambda}^3$. Using the long-time Strichartz estimates proved in \cite{DGG}, we establish polynomial upper bounds for higher Sobolev norms for solutions with small energy.
\end{abstract}

\section{Introduction}
Consider the $3D$ quintic, defocusing nonlinear Schr\"{o}dinger equation,
\begin{equation}\label{nls0}(i\partial_t+\Delta)u=|u|^{4}u,\end{equation} on $\mathbb{R}\times\mathbb{T}_{\lambda}^3$, where $\mathbb{T}_{\lambda}^3$ is a rectangular torus\[\mathbb{T}_\lambda^3=\prod_{j=1}^3\mathbb{R}/(\lambda_i\mathbb{Z}),\quad \lambda=(\lambda_1,\lambda_2,\lambda_3).\] It has been proved (see ~\cite{HTT, IP} and ~\cite{GOW, Str, Str2}) that (\ref{nls0}) is globally well-posed in $H^1(\mathbb{T}_{\lambda}^3)$ with conserved energy\begin{equation}\label{rectnrg}E_{\lambda}[u]= \int_{\mathbb{T}_{\lambda}^3}\left(\frac{1}{2}|\nabla u|^2+\frac{1}{6}|u|^{6}\right)\,\mathrm{d}x,\end{equation} and moreover, if the initial data $u(0)\in H^s(\mathbb{T}_{\lambda}^3)$ for some $s>1$ then $u(t)\in H^s(\mathbb{T}_{\lambda}^3)$ for all time. In this paper we are interested in controlling the possible growth of the $H^s$ norm of $u$ in time for generic irrational tori, in the small energy regime. We prove the following:
 \begin{thm}\label{main} There exists a set $W\subset (\mathbb{R}^+)^3$ with Lebesgue measure zero, such that for each $\lambda\in (\mathbb{R}^+)^3-W$, there exists a constant $\eta>0$ depending on $\lambda$ such that, for any $s> 1$, and any solution $u$ to (\ref{nls0}) with energy $E_{\lambda}[u]\leq\eta^2$ and initial data $u(0)\in H^s(\mathbb{T}_\lambda^3)$, as described in Proposition \ref{locale} below, one has that\[\|u(t)\|_{H^s(\mathbb{T}_{\lambda}^3)}\lesssim\max\big(\|u(0)\|_{H^s(\mathbb{T}_{\lambda}^3)},|t|^{300(s-1)}\big)\] for any time $t\in\mathbb{R}$. Here and below the implicit constants depend on $\lambda$ and $s$, but not on $t$.
\end{thm}
\begin{rem} To the best of the author's knowledge, this is the first result establishing polynomial growth of higher Sobolev norms for an energy critical equation on a compact domain.
\end{rem}
\begin{rem} Note that, for rational tori, where $\lambda_j^2$ are rational multiples of each other, our argument completely fails, and one does not know whether Theorem \ref{main} is still true in this case. This demonstrates the fundamental difference between rational and irrational tori in terms of long-time dispersion properties. (See also \cite{FSWW}, Appendix A, for a discussion of this difference in terms of random data theory.)
\end{rem}
\begin{rem} We can identify some specific elements of $(\mathbb{R}^+)^3-W$. For example, by slightly modifying the proof, one can show that the conclusion of Theorem \ref{main} holds for $\lambda=(\lambda_1,\lambda_2,\lambda_3)$, possibly with different powers of $|t|$, in each of the following two cases:
\begin{enumerate}
\item When $\lambda_j^2$ are algebraic numbers that are linearly independent over $\mathbb{Q}$, for example when $\lambda=(1,\sqrt[4]{2},\sqrt[4]{3})$.
\item When the ratio of two of $\lambda_j^2$ is an irrational number with finite irrationality measure, for example when $\lambda=(1,\pi,e)$.
\end{enumerate}
\end{rem}
\begin{rem} The natural question left open by Theorem \ref{main} is, whether the result remains true for solutions with arbitrarily large energy. It is expected that one needs to combine the tools in the current paper with the profile decomposition techniques (see for example \cite{IP}), but as of now it is still not clear how this can be done.
\end{rem}
\subsection{Backgrounds, and main ideas}  The possible growth of the $H^s\,(s>1)$ norm of solutions to a nonlinear Schr\"{o}dinger equation is linked to the \emph{weak turbulence} phenomenon, where the solution transfers energy from low to high frequencies over time, causing the $H^s$ norm to grow while the energy stays bounded. An upper bound for the $H^s$ norm in time can be seen as a control on how fast this energy transfer can happen.

In general, for a nonlinear Schr\"{o}dinger equation that is locally well-posed in $H^1$, one can easily obtain an exponential upper bound for the $H^s$ norm, \[\|u(t)\|_{H^s}\lesssim e^{C|t|}\|u(0)\|_{H^s},\] by iterating $H^1$ local theory and using preservation of regularity arguments. A breakthrough was made by Bourgain in \cite{B2}, where he used the ``high-low method'' to improve this to a polynomial upper bound, in the case of cubic nonlinearity in $2D$ and $3D$. Further improvements and extensions to other dimensions, nonlinearities, and other dispersive models have since been made by different authors, see ~\cite{B3, B4, CKO, Del, GG, PTV, Soh1, Soh2, Soh3, Staf, Th, Z} and references therein.

Note that, all the above-mentioned results apply for \emph{energy subcritical} equations; this is due to the nature of Bourgain's high-low method. To illustrate, assume for now that we have an energy subcritical nonlinearity $\mathcal{N}(u)$ instead of $|u|^4u$ in (\ref{nls0}), and that $E_{\lambda}[u]\leq\eta^2$. To obtain a polynomial bound for $\|u(t)\|_{H^s}$, it is crucial to prove an inequality of form\begin{equation}\label{iter}\|u(t+1)\|_{H^s}^2-\|u(t)\|_{H^s}^2\lesssim \|u(t)\|_{H^s}^{2(1-\theta)}\end{equation} for some $\theta>0$. Let $t=0$; choosing a parameter $N$, we can make a high-low decomposition $u=v+w$, where \[v(0)=P_{\leq N}u(0),\quad (i\partial_t-\Delta)v=\mathcal{N}(v).\] Intuitively speaking, $v$ consists of frequencies $\lesssim N$ while $w$ consists of frequencies $\gtrsim N$. The energy of $v$ is conserved, so \[\|v(1)\|_{H^s}\lesssim N^{s-1}\|v(1)\|_{H^1}\lesssim N^{s-1}.\] On the other hand, the equation satisfied by $w$ has the form \[(\partial_t-i\Delta)w=\mathcal{N}_1(v)w+\mathcal{N}_2(v)\overline{w}+\mathcal{N}_3(v,w,\overline{w}),\] with $\mathcal{N}_3$ containing at least two factors of $w$. Now, the term $\mathcal{N}_1(v)w$ is consistent with (almost) conservation of $H^s$ norm of $w$, the term $\mathcal{N}_2(v)\overline{w}$ is acceptable due to time non-resonance, and the term $\mathcal{N}_3(v,w,\overline{w})$ is also under control, precisely due to the subcritical nature of $\mathcal{N}$. In fact, by $H^1$ local theory we can bound $v$ and $w$ in $X^{1,b}[0,1]$ for $b=1/2+$ (see Definition \ref{defxsb}), and under the subcritical assumption one has\begin{equation}\label{gain}\bigg\|\int_0^t e^{i(t-t')\Delta}\mathcal{N}_3(v(t'),w(t'),\overline{w(t')})\,\mathrm{d}t'\bigg\|_{X^{s,b}[0,1]}\lesssim O_{\|v\|_{X^{1,b}[0,1]},\|w\|_{X^{1,b}[0,1]}}(1)\cdot\|w\|_{X^{s,b}[0,1]}\|w\|_{X^{\rho,b}[0,1]},\end{equation} where $\rho$ is the critical index of the nonlinearity, which is \emph{strictly less than} $1$. Since $w$ is contains frequencies $\gtrsim N$, this gains a negative power of $N$ for the $\mathcal{N}_3$ term, which translates into an inequality of type\begin{equation}\label{ineq}\|u(1)\|_{H^s}^2\leq\|u(0)\|_{H^s}^2+O(1)(N^{-\theta}\|u(0)\|_{H^s}^2+N^{2(s-1)}),\end{equation} and ultimately implies (\ref{iter}) by optimizing $N$.

The above argument clearly does not work for the energy critical equation (\ref{nls0}). Instead, for generic irrational tori, we will rely on the long-time Strichartz estimate proved in \cite{DGG}, which implies, for example, that\[\sup_{|I|=N^\gamma}\|P_{\geq N}u\|_{L_{t,x}^{10}(I)}\lesssim \eta\] for some $\gamma>0$ uniformly in $N$, if $u$ is a solution to the \emph{linear} Schr\"{o}dinger equation with $E[u]\leq\eta^2$. If this estimate can be shown also for the nonlinear solution $u$, then the norm\footnote{Actually we are using a variation of this norm, measuring $L_{t,x}^{q'}$ norm over time $1$ intervals, and taking $\ell^{q}$ norm for these intervals; see Definition \ref{longtimenorm}.} \[\|u\|_{S_N}:=\sup_{|I|=N^\gamma}\|P_{\geq N}u\|_{L_{t,x}^{10}(I)}\] can be used to substitute the $X^{\rho,b}$ norm appearing in (\ref{gain}) and, instead of gaining a negative power $N^{-\theta}$ over time $1$ as above, allows for a uniform control for the nonlinearity over time $N^{\gamma}$. This implies, as a repacement of (\ref{ineq}), that \[\|u(N^{\gamma})\|_{H^s}^2\lesssim 2\|u(0)\|_{H^s}^2+N^{2(s-1)},\] which again gives polynomial growth. For details see Section \ref{imet}.

It remains to control $\|u\|_{S_N}$ for the nonlinear solution $u$. This is in fact extremely delicate, as we are in the critical setting and any kind of loss cannot be allowed. For large energy solutions it is currently unknown how to achieve this. In the small energy regime, the idea is to use the Duhamel formula and set up a bootstrap argument. For example, we have \[P_Nu(t)=e^{it\Delta}P_Nu(0)-i\int_0^te^{i(t-t')\Delta}P_N(|P_Ku(t')|^4P_Nu(t'))\,\mathrm{d}t'\] if we consider the $|P_Ku|^4P_Nu$ component in the nonlinearity with $K\leq N$. Here, the danger is that we need to control the nonlinearity over time $N^{\gamma}$, while the bootstrap assumption only allows for control of $P_Ku$ over time $K^{\gamma}$. Fortunately we can resort to a trilinear Strichartz estimate (Proposition \ref{trilinear}), which is proved in \cite{HTT} in the context of local well-posedness, and which allows us to gain a power $(K/N)^{1/20}$ provided $K\geq N^{1/100}$. When $K<N^{1/100}$ the corresponding contribution is almost nonresonant, and can thus be controlled using traditional $X^{s,b}$ estimates. We believe this long-time estimate is interesting in its own right, and may be applicable in other situations, for example the random data problem for (\ref{nls0}). For details see Section \ref{longbd}.
\subsection{Notations}\label{notat}We will use $\chi$ to denote general cutoff functions, and $\mathbf{1}_{E}$ the characteristic function of a set $E$. We denote by $C$ or $O(1)$ any constant that may take different values at different places, and write $A \lesssim B$ if $A \leq O(1)B$. These constants will depend on $\lambda$ and $s$, but not on $\eta$. We use the notation $\widetilde{u}$ to denote either $u$ or $\overline{u}$.

After making the change of variables in Section \ref{prep}, we will frequently use Lebesgue norms of form $L_{t}^qL_x^r(I\times\mathbb{T}^3)$, where $I$ is an interval, and will subsequently abbreviate them as $L_{t}^qL_x^r(I)$, and $L_{t}^qL_x^r[a,b]$ if $I=[a,b]$.

We will use capital letters $N,M,K,\cdots$ to denote dyadic numbers in $2^{\mathbb{N}}$. Define as usual the Littlewood-Paley decomposition\begin{equation}\label{lp}1=\sum_{N} P_N,\end{equation} where \[\widehat{P_Nf}(k)=\big[\varphi(k/2N)-\varphi(k/N)\big]\widehat{f}(k),\,N\geq 2;\quad\,\, \widehat{P_1f}(k)=\varphi(k)\widehat{f}(k),\] and $\varphi(y)=\varphi(|y|)$ is a radial smooth function that equals $1$ for $|y|\leq 1$ and equals $0$ for $|y|\geq 2$. Define projections like $P_{\leq N}$ or $P_{\geq N}$ accordingly. For each fixed $N$, define another kind of decomposition \begin{equation}\label{trans}1=\sum_{\mathcal{B}}P_\mathcal{B},\end{equation} where $\mathcal{B}$ runs over the collection of balls of radius $N$ centered at points in $(N\mathbb{Z})^3$, and $P_\mathcal{B}$ is defined by\[\widehat{P_{\mathcal{B}}f}(k)=\psi\bigg(\frac{k-k_0}{N}\bigg)\widehat{f}(k),\quad\text{if }\mathcal{B}=B(k_0,N),k_0\in\mathbb{R}^3,\] with $\psi$ being a fixed compactly supported smooth radial function such that $\psi\equiv 1$ in a neighborhood of $0$, and \[\sum_{k\in\mathbb{Z}^3}\psi(y-k)\equiv 1.\] Note that $P_N$ and $P_{\mathcal{B}}$ are bounded also in $L^p$ spaces. We also need sharp cutoff functions, denoted by $\mathbb{P}_E$ for any set $E\subset\mathbb{Z}^3$, which are defined by \[\widehat{\mathbb{P}_Ef}(k)=\mathbf{1}_E(k)\widehat{f}(k).\] In particular, let $\mathbb{P}_0=\mathbb{P}_{\{(0,0,0)\}}$ and $\mathbb{P}_{\neq 0}=\mathbb{P}_{\mathbb{Z}^3-\{(0,0,0)\}}$.

For each $N$, define the $\mathcal{D}$-multiplier by
\begin{equation}\label{mult}\widehat{\mathcal{D}_Nu}(k)=m(\xi)\widehat{u}(k),\quad m(k)=\theta(k/N),\end{equation} where $\theta(y)=\theta(|y|)$ is a radial smooth function such that $\theta(y)=1$ for $|y|\leq 1$ and $\theta(y)=|y|^{s-1}$ for $|y|\geq 2$. When $N$ is fixed, we will omit the subscript and simply denote $\mathcal{D}_N$ by $\mathcal{D}$. Note that we have the estimate \begin{equation}\label{leibniz}\|\mathcal{D}(fg)\|_{L^p}\lesssim\|\mathcal{D}f\|_{L^q}\|\mathcal{D}g\|_{L^r}\end{equation} for $1\leq p,q,r\leq\infty$ and $1/p=1/q+1/r$, which follows from Leibniz rule and the definition of $\mathcal{D}$.\subsection{Plan of the paper} In Section \ref{prep}, we will collect the tools that will be needed in the proof of Theorem \ref{main}. We will define the spacetime norms in Section \ref{def}, prove relevant linear and nonlinear estimates in Sections \ref{lin0} and \ref{non0}, and review the small data global well-posedness result in Section \ref{local}. In Section \ref{prof} we will prove Theorem \ref{main}; the proof consists of two parts, the $I$-method argument in Section \ref{imet}, and the long-time Strichartz control in Section \ref{longbd}.\section{Preparations}\label{prep} First notice that, by a change of variables, one can reduce (\ref{nls0}) to the equation \begin{equation}\label{nls}(i\partial_t+\Delta_{\beta})u=|u|^{4}u,\end{equation} on $\mathbb{R}\times \mathbb{T}^3$, where $\mathbb{T}^3=(\mathbb{R}/\mathbb{Z})^3$ is the standard square torus, and $\Delta_{\beta}$ is the ``anisotropic'' Laplacian\begin{equation}\label{rect}\Delta_{\beta}=\beta_1\partial_{x_1}^2+\beta_2\partial_{x_2}^2+\beta_3\partial_{x_3}^2,\quad\beta_i=\lambda_i^{2}.\end{equation} The corresponding conserved energy for (\ref{nls}) is \begin{equation}\label{conserve}E[u]:=\int_{\mathbb{T}^3}\bigg(\frac{1}{2}\sum_{i=1}^3\beta_i|\partial_i u|^2+\frac{1}{6}|u|^{6}\bigg)\,\mathrm{d}x.\end{equation}
Now we will focus on the equation (\ref{nls}). For simplicity we will write $\Delta$ instead of $\Delta_\beta$. 
\subsection{Definition of norms}\label{def} Let us recall the definition of the various critical norms as in \cite{HTT}.
\begin{df}[Definition of $U^p$ and $V^p$] Define a partition of $\mathbb{R}$ to be a sequence \[\mathcal{P}=\{t_m\}_{m=0}^M,\quad-\infty<t_0<t_1<\cdots <t_M\leq+\infty.\] Given $1\leq p< \infty$ and a separable Hilbert space $H$, define a $U^p$ atom to be a function $a:\mathbb{R}\to H$ of form\[a(t)=\sum_{m=1}^M\mathbf{1}_{[t_{m-1},t_m)}\phi_{m-1},\] where $\{t_m\}_{m=0}^M$ is a partition of $\mathbb{R}$ and $\phi_m\in H$, and \[\sum_{m=0}^{M-1}\|\phi_m\|_H^p=1.\] Define the space $U^p(\mathbb{R};H)$ by the norm \begin{equation}\label{defu}\|u\|_{U^p(\mathbb{R};H)}=\inf\bigg\{\sum_{j=1}^{\infty}|\gamma_j|:u=\sum_{j=1}^{\infty}\gamma_ja_j,\text{ each $a_j$ is an $U^p$ atom}\bigg\}.\end{equation} Define the space $V^p(\mathbb{R};H)$ by the norm\begin{equation}\label{defv}\|u\|_{V^p(\mathbb{R};H)}=\sup_{\mathcal{P}}\bigg(\sum_{m=1}^M\|u(t_{m})-u(t_{m-1})\|_{H}^p\bigg)^{\frac{1}{p}},\end{equation} where one understands $u(+\infty)=0$ when $t_M=+\infty$. Moreover, we shall restrict to the (closed) subspace of $V^p(\mathbb{R};H)$ consisting only of right continuous functions $u$ such that $\lim_{t\to-\infty}u(t)=0$.
\end{df}
\begin{df}[Definition of spacetime norms] Let $s\in\mathbb{R}$ and $u:\mathbb{R}\times\mathbb{T}^3\to\mathbb{C}$. Define the norms\begin{equation}\label{defust}\|u\|_{U_\Delta^pH^s}=\|e^{-it\Delta_{\beta}}u\|_{U^p(\mathbb{R}:H^s(\mathbb{T}^3))},\end{equation} and \begin{equation}\label{defvst}\|u\|_{V_\Delta^pH^s}=\|e^{-it\Delta_{\beta}}u\|_{V^p(\mathbb{R}:H^s(\mathbb{T}^3))}.\end{equation} Moreover, define
\begin{equation}\label{defx}\|u\|_{X^s}=\bigg(\sum_{k\in\mathbb{Z}^3}\langle k\rangle^{2s}\big\|e^{iQ(k)t}\widehat{u}(t,k)\big\|_{U^2(\mathbb{R}_t;\mathbb{C})}^2\bigg)^{\frac{1}{2}}\end{equation} and \begin{equation}\label{defy}\|u\|_{Y^s}=\bigg(\sum_{k\in\mathbb{Z}^3}\langle k\rangle^{2s}\big\|e^{iQ(k)t}\widehat{u}(t,k)\big\|_{V^2(\mathbb{R}_t;\mathbb{C})}^2\bigg)^{\frac{1}{2}},\end{equation} where $\widehat{u}$ denotes the Fourier transform in space, and \[Q(k)=\beta_1k_1^2+\beta_2k_2^2+\beta_3k_3^2,\quad\text{for } k=(k_1,k_2,k_3).\] For any compact interval $I\subset\mathbb{R}$, define also the local-in-time spaces $X^s(I)$ and $Y^s(I)$ by\[\|u\|_{X^s(I)}=\inf_{v\equiv u\text{ on }I}\|v\|_{X^s},\quad \|u\|_{Y^s(I)}=\inf_{v\equiv u\text{ on }I}\|v\|_{Y^s}.\] When $I=[a,b]$ we also abbreviate $X^s[a,b]$ and $Y^s[a,b]$.
\end{df}In addition to the $X^p$ and $Y^p$ spaces, we will also use the traditional $X^{s,b}$ norms, namely
\begin{equation}\label{defxsb}\|u\|_{X^{s,b}}=\bigg(\sum_{k\in\mathbb{Z}^3}\int_{\mathbb{R}}\langle k\rangle^{2s}\langle \xi+Q(k)\rangle^{2b}|\mathcal{F}_{t,x}u(k,\xi)|^2\,\mathrm{d}\xi\bigg)^{1/2},\quad \|u\|_{X^{s,b}(I)}=\inf_{v\equiv u\text{ on }I}\|v\|_{X^{s,b}}.\end{equation}Finally, we will define a long-time Strichartz norm, which will play a crucial role in the proof of Theorem \ref{main}.
\begin{df}\label{longtimenorm} Let $7/2\leq q\leq 4$, $5\leq q'\leq 12$. For any $N\geq 1$ and any finite interval $J$, define \begin{equation}\label{longnorm}\|u\|_{S_{N,J}^{q,q'}}:=\bigg(\sum_{m\in\mathbb{Z}}\big(N^{\frac{5}{q'}-\frac{1}{2}}\|u\|_{L_{t,x}^{q'}([m,m+1]\cap J)}\big)^{q}\bigg)^{\frac{1}{q}},\end{equation} and  \begin{equation}\label{longnorm2}\|u\|_{S_{N,J}^{q}}=\bigg(\sum_{m\in\mathbb{Z}}\big(\sup_{5\leq q'\leq 12}N^{\frac{5}{q'}-\frac{1}{2}}\|u\|_{L_{t,x}^{q'}([m,m+1]\cap J)}\big)^{q}\bigg)^{\frac{1}{q}}.\end{equation}Notice that \[\|u\|_{S_{N,J}^q}\sim \max\big(\|u\|_{S_{N,J}^{q,6}},\|u\|_{S_{N,J}^{q,12}}\big)\] by H\"{o}lder, and that $\|u\|_{S_{N,J}^q}$ and $\|u\|_{S_{N,J}^{q,q'}}$ are both decreasing in $q$.\end{df}
\subsection{Linear estimates}\label{lin0} In this section we collect the standard linear Strichartz and embedding estimates as in \cite{HTT}, \cite{GOW}, \cite{Str}. We will also prove Proposition \ref{longstr}, which is a consequence of the long-time Strichartz estimates established in \cite{DGG}.
\begin{prop}[Strichartz estimates] \label{linear} Let $I$ be a time interval with length $|I|\lesssim 1$. We have the following estimates.
\begin{enumerate}
\item Homogeneous Strichartz estimates:
\begin{equation}\label{str1}\big\|e^{it\Delta}P_Nf\big\|_{L_{t,x}^q(I)}\lesssim N^{\frac{3}{2}-\frac{5}{q}}\|P_Nf\|_{L^2},\end{equation} for any fixed $q>10/3$. The same estimate holds for $P_{\mathcal{B}}f$ for any ball $\mathcal{B}$ of radius $N$, and an improved estimate\begin{equation}\label{str2}\big\|e^{it\Delta}\mathbb{P}_\mathcal{C}f\big\|_{L_{t,x}^q(I)}\lesssim N^{\frac{3}{2}-\frac{5}{q}}\bigg(\frac{\#\mathcal{C}}{N^3}\bigg)^{\frac{1}{2}-\frac{5}{3q}-\varepsilon}\|\mathbb{P}_{\mathcal{C}}f\|_{L^2}\end{equation} holds for any fixed $\varepsilon>0$ and any set $\mathcal{C}\subset \mathbb{Z}^3$ with diameter not exceeding $N$.
\item Inhomogeneous Strichartz estimates: let $t_0\in I$ and \begin{equation}\label{duham}\mathcal{I}G(t)=\int_{t_0}^t e^{i(t-t')\Delta}G(t')\,\mathrm{d}t'\end{equation} be the Duhamel operator, then for any $q_1>10/3$ and $1\leq q_2<10/7$, one has\begin{equation}\label{inhom}\|\mathcal{I}P_NG\|_{L_{t,x}^{q_1}(I)}\lesssim N^{\frac{5}{q_2}-\frac{5}{q_1}-2}\|P_NG\|_{L_{t,x}^{q_2}(I)}.\end{equation}
\end{enumerate}
\end{prop}
\begin{proof} The scaling invariant Strichatrz estimate (\ref{str1}) is proved in \cite{KV}; moreover one actually has the estimate \begin{equation}\label{lowf}\|e^{it\Delta}P_{\leq N}f\|_{L_{t,x}^q(I)}\lesssim N^{\frac{3}{2}-\frac{5}{q}}\|P_{\leq N}f\|_{L^2}.\end{equation}The corresponding result for $P_{\mathcal{B}}$ follows from (\ref{lowf}) and Galilean symmetry. Now (\ref{str2}) is a consequence of (\ref{str1}), if one interpolates between $q=10/3+\varepsilon/10$ (which follows from (\ref{str1}) for $P_{\mathcal{B}}f$, since $\mathbb{P}_{\mathcal{C}}f=P_{\mathcal{B}}\mathbb{P}_{\mathcal{C}}f$ for some suitable $\mathcal{B}$) and $q=\infty$, and notices that \[\big\|e^{it\Delta}\mathbb{P}_\mathcal{C}f\big\|_{L_{t,x}^{\infty}(I)}\lesssim (\#\mathcal{C})^{1/2}\|\mathbb{P}_{\mathcal{C}}f\|_{L^2},\] which follows from Hausdorff-Young and H\"{o}lder. Finally, (\ref{inhom}) follows from Christ-Kiselev Lemma, the Strichartz estimate \begin{equation}\big\|e^{it\Delta}P_Nf\big\|_{L_{t,x}^{q_1}(I)}\lesssim N^{\frac{3}{2}-\frac{5}{q_1}}\|P_Nf\|_{L^2},\end{equation} and the dual Strichartz estimate\[\left\|\int_Ie^{-it'\Delta}P_NG(t')\,\mathrm{d}t'\right\|_{L^2}\lesssim N^{\frac{5}{q_2}-\frac{7}{2}}\|P_NG\|_{L_{t,x}^{q_2}(I)}.\]
\end{proof}
\begin{prop}[Properties of $X^s$ and $Y^s$ spaces] We have the following estimates.
\begin{enumerate}
\item Embeddings: we have 
\begin{equation}\label{embed1}U^p(\mathbb{R};H)\hookrightarrow V^p(\mathbb{R};H)\hookrightarrow U^q(\mathbb{R};H)\hookrightarrow L^{\infty}(\mathbb{R};H)\end{equation} for any $1\leq p<q<\infty$, and 
\begin{equation}\label{embed2}U_{\Delta}^2H^s\hookrightarrow X^s\hookrightarrow Y^s\hookrightarrow V_{\Delta}^2H^s.\end{equation} In paricular one has
\begin{equation}\label{embed3}\|u(t)\|_{H_x^s}\lesssim \|u\|_{X^s(I)}\end{equation} for any $t\in I$ if $u$ is weakly left continuous in $t$.
\item Linear and Strichartz estimates: for any finite interval $I$ we have
\begin{equation}\label{std1}\|e^{it\Delta }f\|_{X^s(I)}\lesssim \|f\|_{H^s}.\end{equation} Moreover, if $|I|\lesssim1 $ we have the Strichartz estimates \begin{equation}\label{std2}\|P_Nu\|_{L_{t,x}^q(I)}\lesssim N^{\frac{3}{2}-\frac{5}{q}}\|P_Nu\|_{U_{\Delta}^qL^2(I)}\lesssim N^{\frac{3}{2}-\frac{5}{q}}\|P_Nu\|_{X^0(I)},\end{equation} and the same for $P_{\mathcal{B}}u$ if $\mathcal{B}$ is any ball of radius $N$, for any fixed $q>10/3$.
\item Duality: for $s\geq 0$, $T>0$ and any $u$ we have \begin{equation}\label{duality}\|\mathcal{I}u\|_{X^s([t_0,t_0+T))}\leq\sup\bigg\{\bigg|\int_{[t_0,t_0+T)\times\mathbb{T}^3}u(t,x)\cdot\overline{v(t,x)}\,\mathrm{d}x\mathrm{d}t\bigg|:\|v\|_{Y^{-s}([t_0,t_0+T))}=1\bigg\},\end{equation} where the Duhamel operator $\mathcal{I}$ is defined in (\ref{duham}).
\end{enumerate}
\end{prop}\begin{proof}
For the proof of (\ref{embed1}) and (\ref{embed2}) see \cite{HTT}, Remarks $1$ and $2$, Propositions 2.3 and 2.8. As a consequence we have (\ref{embed3}), since by choosing an extension we may assume $u\in X^s$, thus \[\|u\|_{L_t^{\infty}H_x^s}\lesssim  \|u\|_{V_\Delta^2H^s}\lesssim \|u\|_{X^s},\] and (\ref{embed3}) follows using left weak-continuity. For (\ref{std1}) we may assume $I=[a,b)$ by enlarging $I$, and thus $v(t)=\mathbf{1}_{[a,b)}(t)e^{it\Delta}f$ satisfies\[\widehat{v}(t,k)=\mathbf{1}_{[a,b)}(t)\widehat{f}(k),\quad \|\widehat{v}(\cdot,k)\|_{U^2(\mathbb{R};\mathbb{C})}\lesssim|\widehat{f}(k)|,\] so (\ref{std1}) follows from the definition of $X^s$ norm. Next, (\ref{std2}) follows from (\ref{str1}) by repeating the proof of \cite{HTT}, Corollary 3.2, and using the embedding \[X^0\hookrightarrow V_{\Delta}^2L^2\hookrightarrow U_\Delta^qL^2.\] Finally (\ref{duality}) is proved (up to a time translation) in \cite{HTT}, Proposition 2.11, and the same proof applies here with trivial modifications.
\end{proof}
\begin{prop}[Properties of $X^{s,b}$ spaces] We have the following estimates.
\begin{enumerate}
\item For any $s$ and $b$, and any fixed smooth cutoff function $\chi$, we have\begin{equation}\label{xsb01}\|\chi(t)u\|_{X^{s,b}}\lesssim\|u\|_{X^{s,b}};\end{equation} moreover, let \[\mathcal{I}'G(t)=\chi(t-m)\int_m^te^{-i(t-t')}G(t')\,\mathrm{d}t'\] be the smoothly truncated Duhamel operator, then\begin{equation}\label{xsb02}\|\mathcal{I}'G\|_{X^{s,b}}\lesssim\|G\|_{X^{s,b-1}}\end{equation} for any $s$ and any $b\in (1/2,1)$.

\item The Strichartz estimate\begin{equation}\label{strxsb}\|P_Nu\|_{L_{t,x}^q(I)}\lesssim N^{\frac{3}{2}-\frac{5}{q}}\|u\|_{X^{0,b}}\end{equation} holds for any fixed $b>1/2$ and $q>10/3$, where $I$ is any interval of length $|I|\lesssim 1$. The same estimate is true for $P_{\mathcal{B}}u$ if $\mathcal{B}$ is any ball of radius $N$, as well as the improvement for $\mathbb{P}_{\mathcal{C}}u$ as in (\ref{str2}) when $\mathcal{C}$ has diameter not exceeding $N$.
\item For any fixed $\varepsilon>0$ and any fixed smooth cutoff $\chi$, we have
\begin{equation}\label{xsb1}\|\chi(t)u\|_{X^{s,1/4-\varepsilon}}\lesssim \|u\|_{X^s}.\end{equation}
\end{enumerate}
\end{prop}
\begin{proof} (\ref{xsb01}) and (\ref{xsb02}) are standard (the standard proof for (\ref{xsb02}) treats the case $m=0$, but the general case is no different); see for example \cite{Tao}, Section 2.6. Next, (\ref{strxsb}) and estimates for $P_{\mathcal{B}}u$ and $\mathbb{P}_{\mathcal{C}}u$ follows from the corresponding estimates in Proposition \ref{linear} and the following transference principle, which follows from a standard representation formula for $X^{s,b}$ functions (see \cite{KS})
\[\text{If $\|e^{it\Delta}Pf\|_{L_t^qL_x^r(I)}\lesssim \|Pf\|_{X^s}$, then $\|Pu\|_{L_t^qL_x^r(I)}\lesssim\|u\|_{X^{s,b}}$ when $b>1/2$,}\] where $P$ is a certain projection (say $P_N$, $P_{\mathcal{B}}$ or $\mathbb{P}_{\mathcal{C}}$). Now let us prove (\ref{xsb1}); we may assume $s=0$. By separating modes $e^{ik\cdot x}$, we only need to prove \[\|\chi(t)g(t)\|_{H_t^{1/4-\varepsilon}}\lesssim\|g(t)\|_{U^2}\] for any function $g(t)$. In fact, (\ref{xsb1}) follows from applying this to $e^{iQ(k)t}\widehat{u}(k,t)$ for each $k\in\mathbb{Z}^3$, and taking $\ell^2$ norm in $k$, due to the structure of $X^s$ and $X^{s,b}$ norms. By definition of $U^2$, we may assume $g$ is a $U^2$ atom, which is a step function\[g=\sum_{m=1}^M\alpha_m\mathbf{1}_{I_m},\] where $I_m$ are pairwise disjoint intervals and \[\sum_{m=1}^M|\alpha_m|^2=1.\] Since \[\|\chi(t)g\|_{L^2}^2=\sum_{m}|\alpha_m|^2\int_{I_m}\chi^2(t)\,\mathrm{d}t\lesssim1,\] it suffices to prove \begin{equation}\label{fdec}\int_{|\xi|\sim K}|\widehat{\chi(t)g}(\xi)|^2\,\mathrm{d}\xi\lesssim K^{-1/2+\varepsilon}\end{equation} for each $K\geq 1$. Now we decompose $g$ into two parts\[g_1=\sum_{|I_m|\leq K^{-1/2+\varepsilon}}\alpha_m\mathbf{1}_{I_m},\quad g_2=\sum_{|I_m|\geq K^{-1/2+\varepsilon}}\alpha_m\mathbf{1}_{I_m}.\] Then we have\[\|\chi(t)g_1\|_{L^2}^2=\sum_{|I_m|\leq K^{-1/2+\varepsilon}}|\alpha_m|^2\int_{I_m}\chi^2(t)\lesssim K^{-1/2+\varepsilon}.\] When considering $\chi(t)g_2$ we may assume that all these $I_m$ are contained in a fixed interval (the support of $\chi$), so there are at most $O(K^{1/2-\varepsilon})$ intervals in the sum. Therefore\[\|\chi(t)g_2\|_{H^{1/2-\varepsilon}}\lesssim\sum_{|I_m|\geq K^{-1/2+\varepsilon}}|\alpha_m|\cdot\|\chi(t)\mathbf{1}_{I_m}\|_{H^{1/2-\varepsilon}}\lesssim K^{1/4-\varepsilon/2}\sup_{m}\|\chi(t)\mathbf{1}_{I_m}\|_{H^{1/2-\varepsilon}}\] by H\"{o}lder. Since $\|\chi(t)\mathbf{1}_{I}\|_{H^{1/2-\varepsilon}}\lesssim 1$ uniformly for any interval $I$, which follows from the combination of $L^2$ estimate and pointwise decay estimate of the Fourier transform, we get that\[\int_{|\xi|\sim K}|\widehat{\chi(t)g_2}(\xi)|^2\,\mathrm{d}\xi\lesssim K^{-1/2+\varepsilon}\] also. This proves (\ref{fdec}) and thus (\ref{xsb1}).
\end{proof}
\begin{prop}[A long-time Strichartz estimate]\label{longstr} There exists a set $W'\subset(\mathbb{R}^+)^3$ with Lebesgue measure zero, such that for each $\beta=(\beta_1,\beta_2,\beta_3)\in (\mathbb{R}^+)^3-W'$, one has
\begin{equation}\label{strlong}\sup_{|J|=N^\nu}\|e^{it\Delta}P_Nf\|_{S_{N,J}^{q,q'}}\lesssim N\|P_Nf\|_{L^2}.\end{equation} uniformly for $7/2\leq q\leq 4$, $5\leq q'\leq 12$, $0\leq \nu\leq 1/10$, and any $N$. In particular one has \begin{equation}\label{strlong2}\sup_{|J|=N^\nu}\|e^{it\Delta}P_Nf\|_{S_{N,J}^{q}}\lesssim N\|P_Nf\|_{L^2}\end{equation} uniformly in $q$, $\nu$ and $N$.
\end{prop}
\begin{proof} Since \[\nu\leq \frac{1}{10}<\frac{4}{37}\leq \frac{4(3q-10)}{3q+14},\]by the long-time Strichartz estimates proved in \cite{DGG}, we can find a set $W'\subset(\mathbb{R}^+)^3$ with measure zero such that for each $\beta\in(\mathbb{R}^+)^3-W'$, we have \begin{equation}\label{linearresult}\sup_{|J'|\leq 2N^\nu}\|e^{it\Delta}P_Nf\|_{L_{t,x}^{q}(J')}\lesssim N^{\frac{3}{2}-\frac{5}{q}}\|P_Nf\|_{L^2}\end{equation}uniformly in $q$ and $N$. Now Fix $q$, $q'$ and $\nu$.

Given $|J|=N^\nu$, among all $m$ such that $[m,m+1]\cap J\neq\emptyset$, there are at most two such that $[m,m+1]\not\subset J$. For these $m$ one directly uses (\ref{str1}) to bound \[N^{\frac{5}{q'}-\frac{1}{2}}\|e^{it\Delta}P_Nf\|_{L_{t,x}^{q'}([m,m+1]\cap J)}\lesssim N\|P_Nf\|_{L^2}.\] For any other $m$ we have $[m,m+1]\subset J$, let $I=[m,m+1]$ and $\chi$ be a cutoff function such that $\chi\equiv1$ on $I$ and $\chi\equiv0$ outside $2I$. By Gagliardo-Nirenberg inequality, for any function $g(t)$ one has\begin{equation}\label{gag}\|g\|_{L^{q'}(I)}\lesssim \|\chi(t)g\|_{L^{q'}(\mathbb{R})}\lesssim \|\chi(t)g\|_{L^{q}(\mathbb{R})}^{1-\alpha}\|\partial_t(\chi(t)g)\|_{L^{q}(\mathbb{R})}^{\alpha}\lesssim \|g\|_{L^{q}(2I)}^{1-\alpha}\|\partial_tg\|_{L^{q}(2I)}^{\alpha}+\|g\|_{L^{q}(2I)},\end{equation} where $\alpha=\frac{1}{q}-\frac{1}{q'}$. By choosing $g(t)=\|e^{it\Delta}P_Nf\|_{L_x^{q'}}$ and noticing that \[|\partial_tg(t)|\lesssim \|\partial_te^{it\Delta}P_Nf\|_{L_x^{q'}}\lesssim N^2|g(t)|,\] we get that\[\|e^{it\Delta}P_Nf\|_{L_{t,x}^{q'}(I)}\lesssim N^{2\alpha}\|e^{it\Delta}P_Nf\|_{L_{t}^{q}L_x^{q'}(2I)}\lesssim N^{5\alpha}\|e^{it\Delta}P_Nf\|_{L_{t,x}^{q}(2I)}\] by using (\ref{gag}) and Bernstein. Therefore, if $J=[b,b+N^\nu]$ we get \begin{multline}\|e^{it\Delta}P_Nf\|_{S_{N,J}^{q,q'}}^{q}\lesssim N^q\|P_Nf\|_{L^2}^q+\sum_{b\leq m\leq b+N^\nu-1}\big(N^{\frac{5}{q}-\frac{1}{2}}\|e^{it\Delta}P_Nf\|_{L_{t,x}^{q}[m-1,m+2]}\big)^{q}\\\lesssim N^q\|P_Nf\|_{L^2}^q+N^{5-\frac{q}{2}}\|e^{it\Delta}P_Nf\|_{L_{t,x}^{q}[b-1,b+N^\nu+1]}^{q},\end{multline} which then implies (\ref{strlong}) in view of (\ref{linearresult}).\end{proof}
\subsection{Nonlinear estimates}\label{non0} Here we review the multilinear estimates for $X^s$ spaces, which are established successively in \cite{HTT}, \cite{GOW}, \cite{Str}.
\begin{prop}[An improved trilinear Strichartz estimate]\label{trilinear} Let $I$ be an interval with $|I|\lesssim 1$. Then for $N_1\geq N_2\geq N_3\geq 1$, we have\begin{equation}\label{tri}\bigg\|\prod_{j=1}^3e^{it\Delta}P_{N_j}f_j\bigg\|_{L_{t,x}^2(I)}\lesssim\bigg(\frac{N_3}{N_1}+\frac{1}{N_2}\bigg)^{\frac{1}{16}}N_2N_3\prod_{j=1}^3\|P_{N_j}f_j\|_{L^2}.\end{equation} Moreover, we have the corresponding estimates in $Y^0$ norms, namely that\begin{equation}\label{tri2}\bigg\|\prod_{j=1}^3P_{N_j}u_j\bigg\|_{L_{t,x}^2(I)}\lesssim \bigg(\frac{N_3}{N_1}+\frac{1}{N_2}\bigg)^{\frac{1}{16}}N_2N_3\prod_{j=1}^3\|P_{N_j}u_j\|_{Y^0(I)}.\end{equation}
\end{prop}
\begin{proof} The inequality (\ref{tri}) is proved in \cite{Str}, Proposition 4.1 (see also \cite{HTT} and \cite{GOW} for rational and partially rational cases). The constant $1/16$ follows from direct computations. After choosing suitable extensions of $u_j$, (\ref{tri2}) is a consequence of (\ref{tri}), Strichartz estimates (\ref{std2}) and an interpolation lemma for $U^p$ and $V^p$ spaces (see \cite{HTT}, Lemma 2.4 or \cite{GOW}, Lemma 5.3), in the same way as in the proof of Proposition 3.5 in \cite{HTT}. 
\end{proof}
\subsection{Global well-posedness in $H^1$}\label{local} In this section we summarize the existing small energy global well-posedness results for (\ref{nls}).
\begin{prop}\label{locale} Let $\lambda$ and $\eta$ be fixed as in the statement of Theorem \ref{main}. Given any initial data $f\in H^1(\mathbb{T}^3)$ such that $E[f]\leq\eta^2$, there exists a unique solution \[u\in\bigcap_{T>0}X^{1}([-T,T])\] to (\ref{nls}) with initial data $u(0)=f$. The energy $E[u]$ is conserved and one has \[\|u\|_{X^1(I)}\lesssim\eta\] for any time interval $I$ with $|I|\lesssim 1$. Moreover, for any $s>1$, if $f\in H^s(\mathbb{T}^3)$, then $u(t)\in H^s(\mathbb{T}^3)$ for all time, and \[\|u\|_{X^s(I)}\lesssim\|u(t)\|_{H^s(\mathbb{T}^3)}\] for any time interval $I$ with $|I|\lesssim 1$ and any $t\in I$.
\end{prop}
\begin{proof} The following nonlinear estimate is proved in \cite{Str}, for all $s\geq 1$ and $T\leq 1$:\begin{equation}\label{local0}\|\mathcal{I}(\widetilde{u_1}\cdots \widetilde{u_5})\|_{X^s[0,T)}\lesssim\sum_{j=1}^5\|u_j\|_{X^s[0,T)}\prod_{l\neq j}\|u_l\|_{X^1[0,T)}.\end{equation} Choosing $s=1$, clearly (\ref{local0}) implies small data  local well-posedness (up to time $T=1$) in $H^1$, and conservation of energy gives global well-posedness. Moreover local-in-time control for $X^1(I)$ norm is automatic from the proof of local well-posedness, and local-in-time $X^s(I)$ control follows from (\ref{local0}), the known $X^1(I)$ bound and a bootstrap argument. The arguments are standard and we omit the details.
\end{proof}
\section{Proof of Theorem \ref{main}}\label{prof}
\subsection{Increment of energy: the $I$ method}\label{imet}
It is possible to use the high-low method of Bourgain in this subsection; however for technical reasons we have chosen to use the (upside-down) \emph{I-method}, developed by Colliander-Keel-Staffilani-Takaoka-Tao, which is similar to Bourgain's method in spirit. Recall the multiplier $\mathcal{D}=\mathcal{D}_N$ defined in (\ref{mult}).

\begin{prop}\label{incre} 
 Let $N\geq 1$ be fixed. Suppose $u$ is a solution to (\ref{nls}) on $\mathbb{R}\times\mathbb{T}^3$ such that $\| \mathcal{D} u(0) \|_{H^1} \lesssim \eta$, then we have
\begin{multline}\label{increnrg}\big|E[\mathcal{D}u(t)]-E[\mathcal{D}u(0)]\big|\lesssim\eta^2\sum_{K_0}\min\bigg(1,\frac{K_0}{N}\bigg)\\\times\sum_{K_1,K_3,K_3\lesssim K_0}\prod_{j=1}^3\bigg(\frac{K_j}{K_0}\bigg)^{1/12}\cdot\prod_{j=0}^3\sup_{5\leq q'\leq 12} K_j^{\frac{5}{q'}-\frac{1}{2}}\|P_{K_j}\mathcal{D}u\|_{L_{t,x}^{q'}[0,T]}.\end{multline} uniformly for all $0\leq t\leq T\leq 1$.
\end{prop}
\begin{proof} From the assumptions, we know that \[\|u(0)\|_{H^1}\lesssim \eta,\quad \|u(0)\|_{H^s}\lesssim \eta N^{s-1}.\] By Proposition \ref{locale}, one has that\[\|u\|_{X^{1}[-1,2]}\lesssim\eta,\quad \|u\|_{X^{1}[-1,2]}\lesssim \eta N^{s-1},\] which implies $\|\mathcal{D}u\|_{X^{1}[-1,2]}\lesssim\eta$. By considering a suitable extension of $u$, we may assume \begin{equation}\label{dbound}\|\mathcal{D}u\|_{X^{1}}\lesssim\eta.\end{equation}

Now let us compute the time evolution of $E[\mathcal{D}u]$. In fact, one has\[(i\partial_t-\Delta)(\mathcal{D}u)=|\mathcal{D}u|^{4}(\mathcal{D}u)+\mathcal{R},\] where\[\mathcal{R}=\mathcal{D}(|u|^{4}u)-|\mathcal{D}u|^{4}(\mathcal{D}u).\] Now by conservation of energy for (\ref{nls}), one has that
\begin{equation}\begin{aligned}\partial_tE[\mathcal{D}u]&=\int_{\mathbb{T}^3}\left\{\sum_{j=1}^3\beta_j\Re(\partial_j\overline{\mathcal{D}u}\cdot\partial_j(\mathcal{D}u)_t)+\Re(\overline{\mathcal{D}u}\cdot(\mathcal{D}u)_t)\cdot|\mathcal{D}u|^{4}\right\}\,\mathrm{d}x\\
&=\sum_{j=1}^3\beta_j\Im\int_{\mathbb{T}^3}(\partial_j\overline{\mathcal{D}u}\cdot \partial_j\mathcal{R})\,\mathrm{d}x+\Im\int_{\mathbb{T}^3}|\mathcal{D}u|^{4}\overline{\mathcal{D}u}\cdot\mathcal{R}\,\mathrm{d}x.\end{aligned}\end{equation}
Thus, upon integrating in $t$, we reduce to estimating the space-time integrals\begin{equation}\label{estimate1}\int_{[0,T]\times\mathbb{T}^3}|\mathcal{D}u|^{4}\overline{\mathcal{D}u}\cdot\mathcal{R}\,\mathrm{d}x\mathrm{d}t\end{equation} and
\begin{equation}\label{estimate2}\int_{[0,T]\times\mathbb{T}^3}(\partial_j\overline{\mathcal{D}u}\cdot \partial_j\mathcal{R})\,\mathrm{d}x\mathrm{d}t.\end{equation} Let us first consider the term (\ref{estimate1}). This can be written as \[(\ref{estimate1})=\sum_{N_0,\cdots, N_9}\mathcal{M}_{N_0,\cdots,N_9},\] where \begin{multline}\mathcal{M}_{N_0,\cdots,N_9}=\int_{[0,T]\times\mathbb{T}^3}\mathcal{D}P_{N_0}u\cdot\overline{\mathcal{D}P_{N_1}u}\cdot \mathcal{D}P_{N_2}u\cdot \overline{\mathcal{D}P_{N_3}u}\cdot \mathcal{D}P_{N_4}u\\\times\big(\mathcal{D}(P_{N_5}u\cdot\overline{P_{N_6}u}\cdot P_{N_7}u\cdot \overline{P_{N_8}u}\cdot P_{N_9}u)-\mathcal{D}P_{N_5}u\cdot\overline{\mathcal{D}P_{N_6}u}\cdot \mathcal{D}P_{N_7}u\cdot \overline{\mathcal{D}P_{N_8}u}\cdot \mathcal{D}P_{N_9}u\big)\,\mathrm{d}x\mathrm{d}t.\end{multline} Clearly $\mathcal{M}_{N_0,\cdots,N_9}$ vanishes if $N':=\max_{0\leq j\leq 9}N_j\ll N$, and if $N'\gtrsim N$ we will estimate the two terms $\mathcal{M}_{N_0,\cdots,N_9}$ separately. Then, without loss of generality we can assume $N_0\geq\cdots\geq N_9$, so by (\ref{leibniz}) and H\"{o}lder, we have that\[|\mathcal{M}_{N_0,\cdots,N_9}|\lesssim \prod_{j=0}^4\|P_{N_j}\mathcal{D}\widetilde{u}\|_{L_{t,x}^{6}[0,T]}\cdot\prod_{j=5}^9 \|P_{N_j}\mathcal{D}\widetilde{u}\|_{L_{t,x}^{30}[0,T]}.\] By using (\ref{std2}) we know that\[\|P_{N_j}\mathcal{D}\widetilde{u}\|_{L_{t,x}^{30}[0,T]}\lesssim N_j^{4/3}\|P_{N_j}\mathcal{D}u\|_{X^0}\lesssim \eta N_j^{1/3}\] for $5\leq j\leq 9$, as well as \[\|P_{N_4}\mathcal{D}\widetilde{u}\|_{L_{t,x}^{6}[0,T]}\lesssim N_4^{2/3}\|P_{N_4}\mathcal{D}u\|_{X^0}\lesssim \eta N_4^{-1/3},\]thus \[|\mathcal{M}_{N_0,\cdots,N_9}|\lesssim\eta^6\bigg(\frac{N_5\cdots N_9}{N_0\cdots N_4}\bigg)^{1/3}\prod_{j=0}^3 N_j^{1/3}\|P_{N_j}\mathcal{D}\widetilde{u}\|_{L_{t,x}^{6}[0,T]}.\] Since $N_j$ is nonincreasing in $j$, we have\[\sum_{N_4,\cdots, N_9}\bigg(\frac{N_5\cdots N_9}{N_0\cdots N_4}\bigg)^{1/3}\lesssim\bigg(\frac{N_3^3}{N_0N_1N_2}\bigg)^{1/3}\lesssim\bigg(\frac{N_3}{N_0}\bigg)^{1/3}\lesssim\prod_{j=1}^3\bigg(\frac{N_j}{N_0}\bigg)^{1/9}.\] Let $N_j=K_{j}$ for $0\leq j\leq 3$ where $K_0\gtrsim N$, we get that\[|(\ref{estimate1})|\lesssim\sum_{K_0\gtrsim N}\sum_{K_1,K_3,K_3\leq K_0}\prod_{j=1}^3\bigg(\frac{K_j}{K_0}\bigg)^{1/9}\cdot\prod_{j=0}^3 K_j^{1/3}\|P_{K_j}\mathcal{D}\widetilde{u}\|_{L_{t,x}^{6}[0,T]},\] which is acceptable for (\ref{incre}), since $u$ and $\widetilde{u}$ are the same in terms of Lebesgue norms. 

Now let us consider (\ref{estimate2}). Note that \[\partial_j\mathcal{R}=3[\mathcal{D}(|u|^4\partial_ju)-|\mathcal{D}u|^4\mathcal{D}\partial_ju]+2[\mathcal{D}(u^3\overline{u}\partial_j\overline{u})-(\mathcal{D}u)^3\overline{\mathcal{D}u}\mathcal{D}\overline{\partial_ju}]:=3\mathcal{R}_1+2\mathcal{R}_2.\] The estimates for $\mathcal{R}_1$ and $\mathcal{R}_2$ will be much similar, so we will consider $\mathcal{R}_1$ only. By Littlewood-Paley decomposition, we can reduce to studying\begin{multline}\mathcal{N}_{N_0,\cdots, N_5}:=\int_{[0,T]\times\mathbb{T}^3}P_{N_0}\partial_j\overline{\mathcal{D}u}\cdot\big[\mathcal{D}(P_{N_1}\partial_ju\cdot \overline{P_{N_2}u}\cdot P_{N_3}u\cdot \overline{P_{N_4}u}\cdot P_{N_5}u)\\-\mathcal{D}P_{N_1}\partial_ju\cdot \overline{\mathcal{D}P_{N_2}u}\cdot \mathcal{D}P_{N_3}u\cdot \overline{\mathcal{D}P_{N_4}u}\cdot \mathcal{D}P_{N_5}u)\big].\end{multline} Without loss of generality we can assume $N_2\geq N_3\geq N_4\geq N_5$. We now consider two cases:

(1) Suppose $N_2\gtrsim N':=\max(N_0,N_1)$. Note that $N_2\gtrsim N$ since otherwise $\mathcal{N}_{N_0,\cdots, N_5}$ vanishes; using (\ref{leibniz}) and H\"{o}lder one directly estimates\begin{equation}|\mathcal{N}_{N_0,\cdots, N_5}|\lesssim \|P_{N_0}\partial_j\mathcal{D}\widetilde{u}\|_{L_{t,x}^{60/17}[0,T]}\cdot \|P_{N_1}\partial_j\mathcal{D}\widetilde{u}\|_{L_{t,x}^{60/17}[0,T]}\cdot \|P_{N_2}\mathcal{D}\widetilde{u}\|_{L_{t,x}^{60/11}[0,T]}\cdot \prod_{j=3}^5\|P_{N_j}\mathcal{D}\widetilde{u}\|_{L_{t,x}^{12}[0,T]}.\end{equation} By (\ref{std2}) we know that\[\|P_{N_j}\partial_j\mathcal{D}\widetilde{u}\|_{L_{t,x}^{60/17}[0,T]}\lesssim N_j^{1/12}\|P_{N_j}\mathcal{D}u\|_{X^1}\lesssim \eta N_j^{1/12}\] for $0\leq j\leq 1$, thus\[|\mathcal{N}_{N_0,\cdots, N_5}|\lesssim\eta^2\bigg(\frac{N_0N_1N_3N_4N_5}{N_2^5}\bigg)^{1/12} N_2^{5/12}\|P_{N_2}\mathcal{D}\widetilde{u}\|_{L_{t,x}^{60/11}[0,T]}\cdot\prod_{j=3}^5 N_j^{-1/12} \|P_{N_j}\mathcal{D}\widetilde{u}\|_{L_{t,x}^{12}[0,T]}.\] Notice that $\max(N_0,N_1)\lesssim N_2$, we get that\[\sum_{N_0,N_1\lesssim N_2}\bigg(\frac{N_0N_1N_3N_4N_5}{N_2^5}\bigg)^{1/12}\lesssim\bigg(\frac{N_3N_4N_5}{N_2^3}\bigg)^{1/12}.\] Let $K_j=N_{j+2}$ for $0\leq j\leq 3$, we get \[|\text{Case 1}|\lesssim\eta^2\sum_{K_0\gtrsim N}\sum_{K_1,K_3,K_3\leq K_0}\prod_{j=1}^3\bigg(\frac{K_j}{K_0}\bigg)^{1/12}K_0^{5/12}\|P_{K_0}\mathcal{D}\widetilde{u}\|_{L_{t,x}^{60/11}[0,T]}\cdot\prod_{j=1}^3 K_j^{-1/12}\|P_{K_j}\mathcal{D}\widetilde{u}\|_{L_{t,x}^{12}[0,T]},\] which is acceptable for (\ref{incre}).

(2) Suppose $N_2\ll N'$, then $N_0\sim N_1\gg N_2$. We shall use the partition (\ref{trans}) and decompose correspondingly \[P_{N_1}u=\sum_{\mathcal{B}}P_{N_1}P_{\mathcal{B}}u,\] where $\mathcal{B}$ runs balls of radius $N_2$ centered at points of $(N_2\mathbb{Z})^3$. Then we have, by considering Fourier support and orthogonality, that \begin{multline}\mathcal{N}_{N_0,\cdots, N_5}=\sum_{\mathcal{B}}\int_{[0,T]\times\mathbb{T}^3}P_{N_0}P_{2^{10}\mathcal{B}}\partial_j\overline{\mathcal{D}u}\cdot\big[\mathcal{D}(P_{N_1}P_{\mathcal{B}}\partial_ju\cdot \overline{P_{N_2}u}\cdot P_{N_3}u\cdot \overline{P_{N_4}u}\cdot P_{N_5}u)\\-\mathcal{D}P_{N_1}P_{\mathcal{B}}\partial_ju\cdot \overline{\mathcal{D}P_{N_2}u}\cdot \mathcal{D}P_{N_3}u\cdot \overline{\mathcal{D}P_{N_4}u}\cdot \mathcal{D}P_{N_5}u)\big].\end{multline} 

(a) If $N_2\gtrsim N$, then by H\"{o}lder one has \begin{multline}|\mathcal{N}_{N_0,\cdots, N_5}|\lesssim \sum_{\mathcal{B}}\|P_{N_0}P_{2^{10}\mathcal{B}}\partial_j\mathcal{D}\widetilde{u}\|_{L_{t,x}^{60/17}[0,T]}\cdot \|P_{N_1}P_{\mathcal{B}}\partial_j\mathcal{D}\widetilde{u}\|_{L_{t,x}^{60/17}[0,T]}\\\times \|P_{N_2}\mathcal{D}\widetilde{u}\|_{L_{t,x}^{60/11}[0,T]}\cdot \|P_{N_3}\mathcal{D}\widetilde{u}\|_{L_{t,x}^{12}[0,T]}\cdot \|P_{N_3}\mathcal{D}\widetilde{u}\|_{L_{t,x}^{12}[0,T]}\cdot \|P_{N_5}\mathcal{D}\widetilde{u}\|_{L_{t,x}^{12}[0,T]}.\end{multline} Repeating the arguments in part (1) above, using (\ref{std2}) but with $P_{N}$ replaced by $P_{\mathcal{B}}$, we get that\begin{multline}|\mathcal{N}_{N_0,\cdots, N_5}|\lesssim\sum_{\mathcal{B}}\|P_{N_1}P_{\mathcal{B}}\mathcal{D}u\|_{X^1}\|P_{N_0}P_{2^{10}\mathcal{B}}\mathcal{D}u\|_{X^1}\bigg(\frac{N_3N_4N_5}{N_2^3}\bigg)^{1/12}\\\times N_2^{5/12}\|P_{N_2}\mathcal{D}\widetilde{u}\|_{L_{t,x}^{60/11}[0,T]}\cdot\prod_{j=3}^5 N_j^{-1/12} \|P_{N_j}\mathcal{D}\widetilde{u}\|_{L_{t,x}^{12}[0,T]}.  \end{multline} By the definition of $X^1$ norm, and notice that for fixed $N_0$ (or $N_1$) there are at most $O(1)$ choices for $N_1$ (or $N_0$), we know that\[\sum_{N_0,N_1}\sum_{\mathcal{B}}\|P_{N_0}P_{2^{10}\mathcal{B}}\mathcal{D}u\|_{X^1}^2\lesssim\|\mathcal{D}u\|_{X^1}^2\lesssim \eta^2,\] and similarly for the sum with $P_{N_1}P_{\mathcal{B}}$ instead of $P_{N_0}P_{2^{10}\mathcal{B}}$, thus by Cauchy-Schwartz we get\[\sum_{N_0,N_1}|\mathcal{N}_{N_0,\cdots,N_5}|\lesssim\eta^2\bigg(\frac{N_3N_4N_5}{N_2^3}\bigg)^{1/12}N_2^{5/12}\|P_{N_2}\mathcal{D}\widetilde{u}\|_{L_{t,x}^{60/11}[0,T]}\cdot\prod_{j=3}^5 N_j^{-1/12} \|P_{N_j}\mathcal{D}\widetilde{u}\|_{L_{t,x}^{12}[0,T]}.\] Writing $K_j=N_{j+2}$ for $0\leq j\leq 3$, we obtain that \[|\text{Case 2a}|\lesssim\eta^2\sum_{K_0\gtrsim N}\sum_{K_1,K_3,K_3\leq K_0}\prod_{j=1}^3\bigg(\frac{K_j}{K_0}\bigg)^{1/12}K_0^{5/12}\|P_{K_0}\mathcal{D}\widetilde{u}\|_{L_{t,x}^{60/11}[0,T]}\cdot\prod_{j=1}^3 K_j^{-1/12}\|P_{K_j}\mathcal{D}\widetilde{u}\|_{L_{t,x}^{12}[0,T]},\] which is acceptable for (\ref{incre}).

(b) Now suppose $N_2\ll N$, then we must have $N_1\gtrsim N$. By the definition of $\mathcal{D}$, we can simplify \[\mathcal{N}_{N_0,\cdots, N_5}=\sum_{\mathcal{B}}\int_{[0,T]\times\mathbb{T}^3}P_{N_0}P_{2^{10}\mathcal{B}}\partial_j\overline{\mathcal{D}u}\cdot[\mathcal{D},G]P_{N_1}P_{\mathcal{B}}\partial_ju,\] where \[G:=\overline{P_{N_2}u}\cdot P_{N_3}u\cdot \overline{P_{N_4}u}\cdot P_{N_5}u.\] By H\"{o}lder we have\[\|G\|_{L_{t,x}^{30/13}[0,T]}\lesssim \bigg(\frac{N_3N_4N_5}{N_2^5}\bigg)^{1/12}N_2^{5/12}\|P_{N_2}\widetilde{u}\|_{L_{t,x}^{60/11}[0,T]}\cdot\prod_{j=3}^5 N_j^{-1/12} \|P_{N_j}\widetilde{u}\|_{L_{t,x}^{12}[0,T]}.\] Now let $F=P_{N_1}P_{\mathcal{B}}\partial_ju$, then $\widehat{F}(k,t)$ is supported in $|k|\sim N_1$ and $\widehat{G}(k,t)$ is supported in $|k|\lesssim N_2$. For fixed $t$, consider the bilinear expression $[\mathcal{D},G]F$ defined by \[\mathcal{F}[\mathcal{D},G]F(k)=\sum_{m}\chi_1\bigg(\frac{k}{N_1}\bigg)\chi\bigg(\frac{m}{N_2}\bigg)\bigg[\theta\bigg(\frac{k}{N}\bigg)-\theta\bigg(\frac{k-m}{N}\bigg)\bigg]\widehat{F}(k-m)\widehat{G}(m),\] where $\chi_1(z)$ is supported in $|z|\sim 1$, and $\chi_2(z)$ supported in $|z|\lesssim 1$. Since\[\theta\bigg(\frac{k}{N}\bigg)-\theta\bigg(\frac{k-m}{N}\bigg)=\frac{m}{N}\cdot \int_0^1\nabla\theta\bigg(\frac{k-\nu m}{N}\bigg)\,\mathrm{d}\nu,\] and we are in the region where $|k|\sim N_1\gtrsim N$ and $|m|\lesssim N_2\ll N$, we can use Coifman-Meyer theory and transference principle to conclude that\[\|[\mathcal{D},G]F\|_{L_{t,x}^{60/43}[0,T]}\lesssim \frac{N_2}{N}\|F\|_{L_{t,x}^{60/17}[0,T]}\|G\|_{L_{t,x}^{30/13}[0,T]},\] which gives that\begin{multline}|\mathcal{N}_{N_0,\cdots, N_5}|\lesssim\sum_{\mathcal{B}}\|P_{N_0}P_{2^{10}\mathcal{B}}\partial_j\overline{\mathcal{D}u}\|_{L_{t,x}^{60/17}[0,T]}\cdot \frac{N_2}{N}\|F\|_{L_{t,x}^{60/17}[0,T]}\|G\|_{L_{t,x}^{30/13}[0,T]}\\\lesssim\sum_{\mathcal{B}}\frac{N_2^{7/6}}{N}\|G\|_{L_{t,x}^{30/13}[0,T]}\cdot \|P_{N_1}P_{\mathcal{B}}\mathcal{D}u\|_{X^1}\|P_{N_0}P_{2^{10}\mathcal{B}}\mathcal{D}u\|_{X^1}\end{multline}by using (\ref{std2}). Summing over $\mathcal{B}$, $N_0$ and $N_1$ just as in part (2a) above, we get that \[\sum_{N_0,N_1}|\mathcal{N}_{N_0,\cdots,N_5}|\lesssim\eta^2\frac{N_2}{N}\bigg(\frac{N_3N_4N_5}{N_2^3}\bigg)^{1/12}N_2^{5/12}\|P_{N_2}\widetilde{u}\|_{L_{t,x}^{60/11}[0,T]}\cdot\prod_{j=3}^5 N_j^{-1/12} \|P_{N_j}\widetilde{u}\|_{L_{t,x}^{12}[0,T]}.\] Writing $K_j=N_{j+2}$ for $0\leq j\leq 3$, we obtain that \[|\text{Case 2b}|\lesssim\eta^2\sum_{K_0\ll N}\frac{K_0}{N}\sum_{K_1,K_3,K_3\leq K_0}\prod_{j=1}^3\bigg(\frac{K_j}{K_0}\bigg)^{1/12}K_0^{5/12}\|P_{K_0}\widetilde{u}\|_{L_{t,x}^{60/11}[0,T]}\cdot\prod_{j=1}^3 K_j^{-1/12}\|P_{K_j}\widetilde{u}\|_{L_{t,x}^{12}[0,T]},\] which is acceptable for (\ref{incre}).
\end{proof}
\begin{cor}\label{coro} Suppose $\gamma<1/6$. Suppose $u$ is a solution to (\ref{nls}) on $\mathbb{R}\times\mathbb{T}^3$ such that \begin{equation}\label{nrgboot}\sup_{0\leq t\leq T}\| \mathcal{D} u(t) \|_{H^1} \lesssim \eta,\quad T\leq N^{\gamma},\end{equation} then one has \begin{equation}\label{increnrg2}\big|E[\mathcal{D}u(t)]-E[\mathcal{D}u(0)]\big|\lesssim\eta^2\sum_{K}\min\bigg(1,\frac{K}{N}\bigg)^{1/6}\sup_{|J|\leq K^{\gamma},J\subset[0,T]}\|P_K\mathcal{D}u\|_{S_{K,J}^{4}}^4.\end{equation} uniformly for all $0\leq t\leq T$.
\end{cor}
\begin{proof} Using (\ref{incre}), by AM-GM one has\begin{multline}\label{amgm}\sum_{K_0}\min\bigg(1,\frac{K_0}{N}\bigg)\sum_{K_1,K_3,K_3\lesssim K_0}\prod_{j=1}^3\bigg(\frac{K_j}{K_0}\bigg)^{1/12}\cdot\prod_{j=0}^3\sup_{5\leq q'\leq 12} K_j^{\frac{5}{q'}-\frac{1}{2}}\|P_{K_j}\mathcal{D}u\|_{L_{t,x}^{q'}[0,T]}\\\lesssim\sum_{j=0}^3\sum_{K_0;K_1,K_2,K_3\lesssim K_0}\min\bigg(1,\frac{K_0}{N}\bigg)\cdot\bigg(\frac{K_j}{K_0}\bigg)^{1/3}\bigg(\sup_{5\leq q'\leq 12} K_j^{\frac{5}{q'}-\frac{1}{2}}\|P_{K_j}\mathcal{D}u\|_{L_{t,x}^{q'}[0,T]}\bigg)^4.\end{multline} If $j=0$, the above simplifies to\[\sum_{K_0}\min\bigg(1,\frac{K_0}{N}\bigg)\cdot \bigg(\sup_{5\leq q'\leq 12} K_0^{\frac{5}{q'}-\frac{1}{2}}\|P_{K_0}\mathcal{D}u\|_{L_{t,x}^{q'}[0,T]}\bigg)^4;\] if $1\leq j\leq 3$, say $j=1$, then the above simplifies to\[\sum_{K_1}\min\bigg(1,\frac{K_1}{N}\bigg)^{1/3}\cdot \bigg(\sup_{5\leq q'\leq 12} K_1^{\frac{5}{q'}-\frac{1}{2}}\|P_{K_1}\mathcal{D}u\|_{L_{t,x}^{q'}[0,T]}\bigg)^4.\] In any case, for an interval $[m,m+T]$ with $T\leq1$ such that $\|\mathcal{D}u(m)\|_{H^1}\lesssim\eta$, we have that\[|\text{Inc}|\lesssim\eta^2\sum_{K}\min\bigg(1,\frac{K}{N}\bigg)^{1/3}\bigg(\sup_{5\leq q'\leq 12} K^{\frac{5}{q'}-\frac{1}{2}}\|P_{K}\mathcal{D}u\|_{L_{t,x}^{q'}[m,m+T]}\bigg)^4,\] where $\text{Inc}$ denotes the increment of energy $E[\mathcal{D}u(t)]$ on interval $I$. Now, under the assumption (\ref{nrgboot}), the above estimate holds for any $[m,m+1]\cap[0,T]$, thus we have \[|E[\mathcal{D}u(t)]-E[\mathcal{D}u(0)]|\lesssim\eta^2\sum_{K}\min\bigg(1,\frac{K}{N}\bigg)^{1/3}\sum_{m}\bigg(\sup_{5\leq q'\leq 12} K^{\frac{5}{q'}-\frac{1}{2}}\|P_{K}\mathcal{D}u\|_{L_{t,x}^{q'}([m,m+1]\cap[0,T])}\bigg)^4\] for all $0\leq t\leq T$. Since $T\leq N^\gamma$, By dividing $[0,T]$ into intervals of length $\leq K^{\gamma}$ (there will be $\lesssim\max(1,(N/K)^{\gamma})$ of them) and using Definition \ref{longtimenorm}, we can bound\[\sum_{m}\bigg(\sup_{5\leq q'\leq 12} K^{\frac{5}{q'}-\frac{1}{2}}\|P_{K}\mathcal{D}u\|_{L_{t,x}^{q'}([m-1,m+2]\cap[-5,T+5])}\bigg)^4\lesssim \max\bigg(1,\frac{N}{K}\bigg)^\gamma\cdot\sup_{|J|\leq K^{\gamma},J\subset[0,T]}\|P_K\mathcal{D}u\|_{S_{K,J}^{4}}^4.\] Summing over $K$ and using that $\gamma<1/6$, this implies (\ref{increnrg2}).
\end{proof}
\subsection{Long time Strichartz bounds: control of nonlinearity}\label{longbd} In this section we prove the following 
\begin{prop}\label{nonlctrl} Let $\gamma=1/300$. Let $u$ be a solution to (\ref{nls}) on $\mathbb{R}\times\mathbb{T}^3$ such that \begin{equation}\label{nrgboot2}\sup_{t\in[0,T]}\| \mathcal{D} u(t) \|_{H^1} \lesssim \eta,\end{equation} with $T\leq N^\gamma$, and define \begin{equation}\label{nonlinear}A_K=A_K(T):=\sup_{|J|\leq K^{\gamma},J\subset[0,T]}\|P_K\mathcal{D}u\|_{S_{K,J}^{7/2}},\end{equation} then, if $A_K\lesssim 1$ for any $K$, then we have \begin{equation}\label{nonlinear2}A_K\lesssim\eta,\quad \sum_{K\geq N}A_K^2\lesssim\eta^2.\end{equation}
\end{prop}
Before starting the proof, we first recall a ``discrete acausal Gronwall inequality''. For a proof see \cite{Tao}, Theorem 1.19 and Corollary 1.20.
\begin{lem}\label{aux} Suppose $\eta\ll 1$, $\{b_K\}$ and $\{c_K\}$ are two positive sequences satisfying \[b_K\lesssim c_K+\eta\sum_M\min\bigg(\frac{K}{M},\frac{M}{K}\bigg)^{1/5000}b_M,\] and \[\sup_K \frac{b_K}{K^{1/6000}}<\infty,\] then we have\[b_K\lesssim\sup_{M}\min\bigg(\frac{K}{M},\frac{M}{K}\bigg)^{1/6000}c_M\] uniformly for all $K$.
\end{lem}
\begin{proof}[Proof of Proposition \ref{nonlctrl}] Let $u(0)=f$ and \[a_K=\left\{\begin{split}&K\|P_K\mathcal{D}f\|_{L^2},&K\geq N,\\
&\eta,&K< N\end{split}\right.,\] we will first prove that
\begin{equation}\label{double}A_K\lesssim a_K+\eta^5 K^{-1/2000}+\eta\sum_{M}\min\bigg(\frac{K}{M},\frac{M}{K}\bigg)^{1/5000} A_M\end{equation} for any $K$. 

Notice that (\ref{nrgboot2}) implies\begin{equation}\label{localxs}\|\mathcal{D}u\|_{X^1(I)}\lesssim\eta\end{equation} for any interval $I\subset[0,T]$ with $|I|\lesssim1$. To bound $A_K$, by definition, we choose an interval $J\subset[0,T]$ with $|J|\leq K^{\gamma}$. Since $T\leq N^\gamma$, if $K\geq N$ we can assume $J=[0,T]$; if $K< N$, by translation, we will also assume that the left endpoint of $J$ is $0$ (see Remark \ref{transl} after the proof).

Now that $J=[0,T']$ with $T'\leq T\leq N^\gamma$, we start with the evolution equation satisfied by $P_K\mathcal{D}u$, namely\[(i\partial_t+\Delta)P_K\mathcal{D}u=P_K\mathcal{D}(|u|^4u).\] By Littlewood-Paley decomposition, we have \begin{align}(i\partial_t+\Delta)P_K\mathcal{D}u&=\sum_{K_1,\cdots, K_5} P_K\mathcal{D}(P_{K_1}u\cdot \overline{P_{K_2}u}\cdot P_{K_3}u\cdot\overline{P_{K_4}u}\cdot P_{K_5}u)\nonumber\\
\label{term1}&=\sum_{\text{at least two $K_j\gtrsim K$}} P_K\mathcal{D}(P_{K_1}u\cdot \overline{P_{K_2}u}\cdot P_{K_3}u\cdot\overline{P_{K_4}u}\cdot P_{K_5}u)\\
\label{term2}&+\sum_{j=1}^5\sum_{\substack{K_j\sim K\\
K_l\ll K(l\neq j)}}P_K\mathcal{D}(P_{K_1}u\cdot \overline{P_{K_2}u}\cdot P_{K_3}u\cdot\overline{P_{K_4}u}\cdot P_{K_5}u).\end{align}Next, we fix $\alpha_1=1/100$, denote $(\ref{term1})=\mathcal{N}_1$, and further decompose (\ref{term2}) using symmetry, as follows:
\begin{align}(\ref{term2})&=\bigg\{2\sum_{\substack{K_2\sim K\\
K^{\alpha_1}\lesssim\max_{l\neq 2}K_l\ll K}}P_K\mathcal{D}(P_{K_1}u\cdot \overline{P_{K_2}u}\cdot P_{K_3}u\cdot\overline{P_{K_4}u}\cdot P_{K_5}u)\nonumber\\
\label{term3}&\qquad\!\!+3\sum_{\substack{K_1\sim K\\
K^{\alpha_1}\lesssim\max_{l\neq 1}K_l\ll K}}P_K\mathcal{D}(P_{K_1}u\cdot \overline{P_{K_2}u}\cdot P_{K_3}u\cdot\overline{P_{K_4}u}\cdot P_{K_5}u)\bigg\}\\
\label{term3+}&+2\sum_{\substack{K_2\sim K\\
\max_{l\neq 2}K_l\ll K^{\alpha_1}}}P_K\mathcal{D}(P_{K_1}u\cdot \overline{P_{K_2}u}\cdot P_{K_3}u\cdot\overline{P_{K_4}u}\cdot P_{K_5}u)\\
\label{term4}&+3\sum_{\substack{K_1\sim K\\
\max_{l\neq 1}K_l\ll K^{\alpha_1}}}P_K\mathcal{D}(P_{K_1}u\cdot \mathbb{P}_{\neq 0}(\overline{P_{K_2}u}\cdot P_{K_3}u\cdot\overline{P_{K_4}u}\cdot P_{K_5}u))\\
\label{term6}&+3P_K\mathcal{D}u\cdot\mathbb{P}_0\bigg(\sum_{K_2,\cdots,K_5\ll K^{\alpha_1}}\overline{P_{K_2}u}\cdot P_{K_3}u\cdot\overline{P_{K_4}u}\cdot P_{K_5}u\bigg).
\end{align} We denote \[(\ref{term3})=\mathcal{N}_2,\quad (\ref{term3+})=\mathcal{N}_3,\quad (\ref{term4})=\mathcal{N}_4,\] and \[3\mathbb{P}_0\bigg(\sum_{K_2,\cdots,K_5\ll K^{\alpha_1}}\overline{P_{K_2}u}\cdot P_{K_3}u\cdot\overline{P_{K_4}u}\cdot P_{K_5}u\bigg)=\omega(t).\] Then we have
\begin{equation}\label{neweqn}(\partial_t+i\Delta)P_K\mathcal{D}u=\mathcal{N}_1+\mathcal{N}_2+\mathcal{N}_3+\mathcal{N}_4+\omega(t)\cdot P_K\mathcal{D}u.\end{equation} Define \begin{equation}\label{defv}v(t)=e^{-i\Omega(t)}P_K\mathcal{D}u(t),\quad \Omega(t)=\int_0^t\omega(t')\,\mathrm{d}t',\end{equation} then we have\[(\partial_t+i\Delta)v=\mathcal{N}_1'+\mathcal{N}_2'+\mathcal{N}_3'+\mathcal{N}_4',\] where $\mathcal{N}_j':=e^{-i\Omega(t)}\mathcal{N}_j$, and thus by Duhamel's formula
\begin{equation}\label{newduham}v(t)=e^{it\Delta}P_K\mathcal{D}f-i\int_0^t e^{i(t-t')\Delta}\big(\mathcal{N}_1'(t')+\mathcal{N}_2'(t')+\mathcal{N}_3'(t')+\mathcal{N}_4'(t')\big)\,\mathrm{d}t'.\end{equation} For $0\leq t\leq T'$, write \[v(t)=v_{\mathrm{lin}}(t)-i\sum_{0\leq m\leq T'}(v_m'+v_m'')(t),\] where $v_{\mathrm{lin}}(t)=e^{it\Delta}P_K\mathcal{D}f$, and \[v_m'(t)=\mathbf{1}_{[m,m+1)}(t)\int_m^t e^{i(t-t')\Delta}\big(\mathcal{N}_1'(t')+\mathcal{N}_2'(t')+\mathcal{N}_3'(t')+\mathcal{N}_4'(t')\big)\,\mathrm{d}t',\] and \[v_m''(t)=\mathbf{1}_{[m+1,+\infty)\cap J}(t)\cdot e^{it\Delta}g_m,\] where \[g_m=\int_m^{m+1}e^{-it'\Delta}\big(\mathcal{N}_1'(t')+\mathcal{N}_2'(t')+\mathcal{N}_3'(t')+\mathcal{N}_4'(t')\big)\,\mathrm{d}t'.\] Note that, when $[m,m+1]\not\subset[0,T']$ (this can happen for at most two $m$), we should replace the interval $[m,m+1]$ by $[m,m+1]\cap [0,T']$. By Proposition \ref{longstr}, we know that \[\|v_{\mathrm{lin}}\|_{S_{K,J}^{7/2}}\lesssim K\|P_K\mathcal{D}f\|_{L^2}\lesssim a_K;\] next let us estimate $v_m'$ and $v_m''$ for each fixed $m$. We will consider the contributions from the terms $\mathcal{N}_j'$ separately; when $j$ is fixed, we will denote by $\mathcal{L}_{K_1,\cdots,K_5}$ the contribution to $\mathcal{N}_j$ from the choice of $(K_1,\cdots,K_5)$, and define $\mathcal{L}_{K_1,\cdots,K_5}'$, $v_{m,K_1,\cdots,K_5}'$,$v_{m,K_1,\cdots,K_5}''$ and $g_{m,K_1,\cdots,K_5}$ accordingly. Since these functions are determined by the value of $u$ on $[m,m+1]$, and $\|\mathcal{D}u\|_{X^1[m-1,m+2]}\lesssim\eta$, by choosing an extension of $u$, in considering these terms we may assume $\|\mathcal{D}u\|_{X^1}\lesssim \eta$.

(1) The term $\mathcal{N}_1'$. We may assume $K_1\geq\cdots\geq K_5$, and $K\lesssim K_2$. By Proposition \ref{longstr} we know\[\|v_{m,K_1,\cdots,K_5}''\|_{S_{K,J}^{7/2}}\lesssim K\|g_{m,K_1,\cdots,K_5}\|_{L^2},\] and by dual Strichartz we know\[K\|g_{m,K_1,\cdots,K_5}\|_{L^2}\lesssim K^{13/12}\|\mathcal{L}_{K_1,\cdots,K_5}'\|_{L_{t,x}^{60/43}[m,m+1]}= K^{13/12}\|\mathcal{L}_{K_1,\cdots,K_5}\|_{L_{t,x}^{60/43}[m,m+1]},\] and by (\ref{inhom}) we know \begin{multline}\|v_{m,K_1,\cdots,K_5}'\|_{S_{K,J}^{7/2}}\lesssim\sup_{5\leq q\leq 12}K^{\frac{5}{q}-\frac{1}{2}}\|v_{m,K_1,\cdots,K_5}'\|_{L_{t,x}^q[m,m+1]}\\\lesssim K^{13/12}\|\mathcal{L}_{K_1,\cdots,K_5}'\|_{L_{t,x}^{60/43}[m,m+1]}= K^{13/12}\|\mathcal{L}_{K_1,\cdots,K_5}\|_{L_{t,x}^{60/43}[m,m+1]}.\end{multline} Moreover by H\"{o}lder and (\ref{leibniz}) we have\[\begin{split}K^{13/12}\|\mathcal{L}_{K_1,\cdots,K_5}\|_{L_{t,x}^{60/43}[m,m+1]}
&\lesssim K^{13/12}K_1^{-5/12}(K_3K_4K_5)^{1/12}\cdot K_1^{5/12}\|P_{K_1}\mathcal{D}\widetilde{u}\|_{L_{t,x}^{60/11}[m,m+1]}\\&\times \|P_{K_2}\mathcal{D}\widetilde{u}\|_{L_{t,x}^{60/17}[m,m+1]}\cdot\prod_{j=3}^5K_j^{-1/12}\|P_{K_j}\mathcal{D}\widetilde{u}\|_{L_{t,x}^{12}[m,m+1]}.\end{split}\] By (\ref{std2}) we have \[\|P_{K_2}\mathcal{D}\widetilde{u}\|_{L_{t,x}^{60/17}[m,m+1]}\lesssim\eta K_2^{1/12}\|P_{K_2}\mathcal{D}u\|_{X^0}\lesssim \eta K_2^{-11/12},\] thus summing over $K_2\gtrsim K$ we get \[\begin{split}K^{13/12}\sum_{K_2}\|\mathcal{L}_{K_1,\cdots,K_5}\|_{L_{t,x}^{60/43}[m,m+1]}
&\lesssim \eta K^{1/6}K_1^{-5/12}(K_3K_4K_5)^{1/12}\\&\times K_1^{5/12}\|P_{K_1}\mathcal{D}\widetilde{u}\|_{L_{t,x}^{60/11}[m,m+1]}\cdot\prod_{j=3}^5K_j^{-1/12}\|P_{K_j}\mathcal{D}\widetilde{u}\|_{L_{t,x}^{12}[m,m+1]}.\end{split}\] Notice that the coefficient\[K^{1/6}K_1^{-5/12}(K_3K_4K_5)^{1/12}\lesssim\bigg(\frac{K}{K_1}\bigg)^{1/48}\cdot\prod_{j=3}^5\min\bigg(\frac{K}{K_j},\frac{K_j}{K}\bigg)^{1/48}\] under the assumption $K\lesssim K_1$ and $K_j\lesssim K_1$, summing over $(K_1,K_3,K_4,K_5)$ and using AM-GM as in (\ref{amgm}), we concluse that\begin{equation}\label{term1}\|v_m'\|_{S_{K,J}^{7/2}}+\|v_m''\|_{S_{K,J}^{7/2}}\lesssim\eta\sum_{M}\min\bigg(\frac{M}{K},\frac{K}{M}\bigg)^{1/12}\sup_{5\leq q\leq 12}M^{\frac{5}{q}-\frac{1}{2}}\|P_M\mathcal{D}u\|_{L_{t,x}^q[m,m+1]}^4.\end{equation}

(2) The term $\mathcal{N}_2'$. By Proposition \ref{longstr} and (\ref{std2}) we have that\begin{equation}\label{gain1}\|v_{m,K_1,\cdots,K_5}'\|_{S_{K,J}^{7/2}}\lesssim\sup_{5\leq q\leq 12}K^{\frac{5}{q}-\frac{1}{2}}\|v_{m,K_1,\cdots,K_5}'\|_{L_{t,x}^q[m,m+1]}\lesssim\|\mathcal{I}(\mathcal{L}_{K_1,\cdots,K_5}\cdot e^{-i\Omega(t)})\|_{X^1[m,m+1)}\end{equation} and \begin{equation}\label{gain2}\|v_{m,K_1,\cdots,K_5}''\|_{S_{K,J}^{7/2}}\lesssim K\|g_{m,K_1,\cdots,K_5}\|_{L^2}\lesssim \|\mathcal{I}(\mathcal{L}_{K_1,\cdots,K_5}\cdot e^{-i\Omega(t)})\|_{X^1[m,m+1)},\end{equation} where $\mathcal{I}$ is the Duhamel operator\[\mathcal{I}G(t)=\int_m^t e^{i(t-t')\Delta}G(t')\,\mathrm{d}t'.\] Since we will not distinguish between $u$ and $\overline{u}$, we will only consider the terms\[\sum_{\substack{K_1\sim K\\
K^{\alpha_1}\lesssim K_2\ll K}}\mathcal{L}_{K_1,\cdots,K_5}\]from $\mathcal{N}_2$, and assume $K_2\geq\cdots\geq K_5$. By (\ref{duality}) we have that\[\|\mathcal{I}(\mathcal{L}_{K_1,\cdots,K_5}\cdot e^{-i\Omega(t)})\|_{X^1[m,m+1)}\lesssim K\int_m^{m+1}\int_{\mathbb{T}^3}e^{-i\Omega(t)}\mathcal{L}_{K_1,\cdots,K_5}\cdot\overline{v}\,\mathrm{d}x\mathrm{d}t,\] where $v$ is some function such that $\|v\|_{Y^0[m,m+1)}=1$. By choosing an extension we may assume $\|v\|_{Y^0}\lesssim 1$; we could move the factor $P_K\mathcal{D}$ in the expression of $\mathcal{L}_{K_1,\cdots,K_5}$ to $v$, and then use Proposition \ref{trilinear} to bound\[\begin{split}\|\mathcal{I}(\mathcal{L}_{K_1,\cdots,K_5}\cdot e^{-i\Omega(t)})\|_{X^1[m,m+1)}&\lesssim K\|e^{-i\Omega(t)}\|_{L_{t,x}^{\infty}}\|P_{K_1}\widetilde{u}P_{K_3}\widetilde{u}P_{K_4}\widetilde{u}\|_{L^2[m,m+1)}\cdot\|P_{K}\mathcal{D}\widetilde{v}P_{K_5}\widetilde{u}P_{K_5}\widetilde{u}\|_{L^2[m,m+1)}\\
&\lesssim \bigg(\frac{K_5}{K}+\frac{1}{K_2}\bigg)^{1/16}K\|P_{K_1}u\|_{Y^0}\|P_K\mathcal{D}v\|_{Y^0}\cdot\prod_{j=2}^5 K_j\|P_{K_j}u\|_{Y^0}.\end{split}\] Now since \[K_j\|P_{K_j}u\|_{Y^0}\lesssim \|P_{K_j}u\|_{Y^1}\lesssim \|P_{K_j}u\|_{X^1}\lesssim\|\mathcal{D}u\|_{X^1}\lesssim\eta\] for $2\leq j\leq 5$, and similarly \[K\|P_{K_1}u\|_{Y^0}\|P_K\mathcal{D}v\|_{Y^0}\lesssim K\max(1,(K/N)^{s-1}) \|P_{K_1}u\|_{Y^0}\|P_Kv\|_{Y^0}\lesssim K_1 \|P_{K_1}\mathcal{D}u\|_{Y^0}\|P_Kv\|_{Y^0}\lesssim\eta,\] this gives that
\begin{equation}\label{gain3}\|\mathcal{I}(\mathcal{L}_{K_1,\cdots,K_5}\cdot e^{-i\Omega(t)})\|_{X^1[m,m+1)}\lesssim\eta^5 \bigg(\frac{K_5}{K}+\frac{1}{K_2}\bigg)^{1/16}\lesssim\eta^5K^{-\alpha_1/16}+\eta^5\bigg(\frac{K_5}{K}\bigg)^{1/16}.\end{equation} On the other hand, choose \[\frac{1}{q_0}=\frac{7}{10}+\delta,\quad \frac{1}{q_1}=\frac{3}{10}+3\delta,\quad \frac{1}{q_2}=\frac{1}{10}+\delta\] with $\delta=\frac{1}{6000}$, by (\ref{inhom}) and dual Strichartz we have the (tame) estimate\[\sup_{5\leq q\leq 12}K^{\frac{5}{q}-\frac{1}{2}}\|v_{m,K_1,\cdots,K_5}'\|_{L_{t,x}^q[m,m+1]}+K\|g_{m,K_1,\cdots,K_5}\|_{L^2}\lesssim K^{1+5\delta}\|\mathcal{L}_{K_1,\cdots,k_5}\|_{L_{t,x}^{q_0}[m,m+1]},\] which by H\"{o}lder, (\ref{leibniz}) and (\ref{std2}), is bounded by\[\begin{split}K^{1+5\delta}\|\mathcal{L}_{K_1,\cdots,k_5}\|_{L_{t,x}^{q_0}[m,m+1]}&\lesssim K^{1+5\delta}\|P_{K_1}\widetilde{u}\|_{L_{t,x}^{q_1}[m,m+1]}\cdot\prod_{j=2}^5\|P_{K_j}\widetilde{u}\|_{L_{t,x}^{q_2}}\\
&\lesssim \eta K^{1+5\delta}K_1^{15\delta-1}(K_2\cdots K_5)^{-5\delta}\prod_{j=2}^5\sup_{5\leq q\leq 12}K_j^{\frac{5}{q}-\frac{1}{2}}\|P_{K_j}\widetilde{u}\|_{L_{t,x}^{q_2}[m,m+1]}\\
&\lesssim \eta\bigg(\frac{K}{K_5}\bigg)^{20\delta}\prod_{j=2}^5\sup_{5\leq q\leq 12}K_j^{\frac{5}{q}-\frac{1}{2}}\|P_{K_j}\widetilde{u}\|_{L_{t,x}^{q_2}[m,m+1]}.\end{split}\] Interpolating this with (\ref{gain3}) and using also (\ref{gain1}) and (\ref{gain2}), we get that\[\|v_{m,K_1,\cdots,K_5}'\|_{S_{K,J}^{7/2}}\lesssim\eta^5K^{-\alpha_1/16}+\eta\bigg(\frac{K_5}{K}\bigg)^{1/256}\prod _{j=2}^5\sup_{5\leq q\leq 12}K_j^{\frac{5}{q}-\frac{1}{2}}\|P_{K_j}\widetilde{u}\|_{L_{t,x}^{q_2}[m,m+1]}^{7/8},\] and the same for $v_{m,K_1,\cdots,K_5}''$. Summing over $K_2,\cdots,K_5\lesssim K$ and using AM-GM as in case (1) above, we get that \begin{equation}\label{term2}\|v_m'\|_{S_{K,J}^{7/2}}+\|v_m''\|_{S_{K,J}^{7/2}}\lesssim\eta^5K^{-\alpha_1/20}+\eta\sum_{M}\min\bigg(\frac{M}{K},\frac{K}{M}\bigg)^{1/256}\sup_{5\leq q\leq 12}M^{\frac{5}{q}-\frac{1}{2}}\|P_M\mathcal{D}u\|_{L_{t,x}^q[m,m+1]}^{7/2}.\end{equation}

(3) Terms $\mathcal{N}_3'$ and $\mathcal{N}_4'$. These terms will be estimated using traditional $X^{s,b}$ spaces, combined with resonance analysis and the improved estimate (\ref{str2}) for thin slices. Note that, in terms of resonance $\mathcal{N}_3$ is better then $\mathcal{N}_4$ due to the choice of $u$ and $\overline{u}$, since this makes the resonance factor $Q(k)+Q(k_1)+\cdots$ instead of $Q(k)-Q(k_1)+\cdots$ where $k$ and $k_1$ are the top two frequencies. Thus we will focus on $\mathcal{N}_4'$. We will assume $K_2\geq\cdots \geq K_5$; by (\ref{xsb1}), we have that $\|\mathcal{D}u\|_{X^{1,1/4-\varepsilon}}\lesssim\eta$ for any fixed $\varepsilon>0$.

We shall prove that \[\|v_{m,K_1,\cdots,K_5}'\|_{S_{K,J}^{7/2}}+\|v_{m,K_1,\cdots,K_5}''\|_{S_{K,J}^{7/2}}\lesssim \eta^5K^{-1/1000}\] for any $(K_1,\cdots,K_5)$, so that the logarithmic factor coming from summation in $(K_1,\cdots,K_5)$ can be omitted. Consider the function \[\rho(t,x)=\mathbb{P}_{\neq 0}\bigg(\prod_{j=2}^5P_{K_j}\widetilde{u}\bigg).\] Note that $K_2\ll K^{\alpha_1}$, we use the equation (\ref{nls}) to compute, for any interval $I$ containing $m$ with $|I|\lesssim 1$, that \[\begin{split}\|\partial_t\rho\|_{L_{t,x}^2(I)}&\lesssim\|P_{\leq K^{\alpha_1}}u\|_{L_{t,x}^{\infty}}^3\cdot\bigg(K^{2\alpha_1}\|P_{\leq K^{\alpha_1}}u\|_{L_{t,x}^2(I)}+\|P_{\leq K^{\alpha_1}}(|u|^4u)\|_{L_{t,x}^2(I)}\bigg)\\
&\lesssim \eta^3K^{3\alpha_1/2}(K^{2\alpha_1}\eta+K^{3\alpha_1/2}\|u\|_{L_t^\infty L_x^5(I)}^5)\lesssim \eta^4K^{4\alpha_1}.\end{split}\] This implies $\|\rho\|_{H_t^1H_x^2}\lesssim\eta^4K^{7\alpha_1}$, therefore by inserting a suitable time cutoff $\chi(t-m)$ we can write \[\chi(t-m)\rho(t,x)=\chi(t-m)\sum_{|k|\leq K^{\alpha_1}}\int_{\mathbb{R}}d(k,\xi)e^{i(k\cdot x+t\xi)}\,\mathrm{d}\xi,\] with\begin{equation}\label{verylow}\sum_{|k|\leq K^{\alpha_1}}\int_{\mathbb{R}}\langle \xi\rangle^{1/3}|d(k,\xi)|\,\mathrm{d}\xi\lesssim\eta^4K^{7\alpha_1}.\end{equation} Similarly, since we know\[\|\partial_te^{-i\Omega(t)}\|_{L^2(I)}=\|\omega(t)\|_{L^2(I)}\lesssim\|P_{\leq K^{\alpha_1}}u\|_{L_{t}^{\infty}L_x^4}^4\lesssim\eta^4,\] we can write\begin{equation*}\chi(t-m)e^{-i\Omega(t)}=\chi(t-m)\int_{\mathbb{R}}y(\xi)e^{it\xi}\,\mathrm{d}\xi,\end{equation*} with \begin{equation}\label{verylow2}\int_{\mathbb{R}}\langle \xi\rangle^{1/3}|y(\xi)|\,\mathrm{d}\xi\lesssim1.\end{equation} Therefore, with fixed $K_1,\cdots,K_5$, we can reduce to estimating the functions\[h_{k,\xi}(t)=\langle \xi\rangle^{-1/3}\chi(t-m)\int_m^t e^{i(t-t')\Delta}\mathcal{D}\big(e^{i(k\cdot x+\xi t)}P_{K_1}u(t,x)\big)\,\mathrm{d}t',\] where $\chi$ is a suitable cutoff function as above, and have that\begin{multline}\|v_{m,K_1,\cdots,K_5}'\|_{S_{K,J}^{7/2}}+\|v_{m,K_1,\cdots,K_5}''\|_{S_{K,J}^{7/2}}\\\lesssim\eta^4K^{7\alpha_1}\sup_{|k|\lesssim K^{\alpha_1},\xi\in\mathbb{R}}\big(\sup_{5\leq q\leq 12}K^{\frac{5}{q}-\frac{1}{2}}\|h_{k,\xi}\|_{L_{t,x}^{q}[m,m+1]}+\|e^{i(t-m-1)\Delta}h_{k,\xi}(m+1)\|_{S_{K,J}^{7/2}}\big).\end{multline} Notice that by (\ref{embed3}), (\ref{strxsb}) and (\ref{strlong2}),\[\sup_{5\leq q\leq 12}K^{\frac{5}{q}-\frac{1}{2}}\|h_{k,\xi}\|_{L_{t,x}^{q}[m,m+1]}+\|e^{i(t-m-1)\Delta}h_{k,\xi}(m+1)\|_{S_{J,K}^{7/2}}\lesssim \|h_{k,\xi}\|_{X^{1,3/4}},\] and that by (\ref{xsb02}), we have \[\langle \xi\rangle^{1/3}\|h_{k,\xi}\|_{X^{1,3/4}}\lesssim \|\nabla\mathcal{D}P_{K_1}u\|_{L_{t,x}^2}\lesssim\|\mathcal{D}P_{K_1}u\|_{X^{1,1/5}}\lesssim\eta,\] we see that \[\eta^4K^{7\alpha_1}\sup_{5\leq q\leq 12}K^{\frac{5}{q}-\frac{1}{2}}\|h_{k,\xi}\|_{L_{t,x}^{q}[m,m+1]}+\|e^{i(t-m-1)\Delta}h_{k,\xi}(m+1)\|_{S_{J,K}^{7/2}}\lesssim \eta^5 K^{-3\alpha_1}\] if $|\xi|\gtrsim K^{30\alpha_1}$. 

Now if $|\xi|\lesssim K^{30\alpha_1}$, we shall decompose $P_{K_1}u=\mathbb{P}_{\mathcal{C}}u+P_{K_1}\mathbb{P}_{\mathcal{C}'}u$, where \[\mathcal{C}=\bigcup_{0<|k|\lesssim K^{\alpha_1}}\big\{n\in\mathbb{Z}^3:|n|\sim K_1,|Q(n+k)-Q(n)|\lesssim K^{60\alpha_1}\big\}\] and $\mathcal{C}'=\mathbb{Z}^3-\mathcal{C}$. Clearly we have $\#\mathcal{C}\lesssim K^{2+40\alpha_1}$. For the term $\mathbb{P}_\mathcal{C}u$, denote its contribution to $h_{k,\xi}$ by $h_{k,\xi}'$, then we have $\|h_{k,\xi}'\|_{X^{1,3/4}}\lesssim\eta$ as above, and moreover the spatial Fourier transform $\widehat{h_{k,\xi}}$ is supported in a translate of $\mathcal{C}$, so by (\ref{str2}) and the corresponding $X^{s,b}$ estimate, we have (note that, for the $S_{J,K}^{7/2}$ norm, we will reduce it to $\lesssim K^{\gamma}$ intervals of length $1$, losing a factor $K^{\gamma}$ in the process)\begin{multline}\eta^4K^{7\alpha_1}\sup_{5\leq q\leq 12}K^{\frac{5}{q}-\frac{1}{2}}\|h_{k,\xi}'\|_{L_{t,x}^{q}[m,m+1]}+\|e^{i(t-m-1)\Delta}h_{k,\xi}'(m+1)\|_{S_{J,K}^{7/2}}\\\lesssim\eta^4K^{7\alpha_1+\gamma}\bigg(\frac{K^{2+40\alpha_1}}{K^3}\bigg)^{\frac{1}{6}}\|h_{k,\xi}'\|_{X^{1,3/4}}\lesssim \eta^5 K^{-1/6+15\alpha_1},\end{multline} using the fact that $\gamma<\alpha_1$.

Finally, for the term $P_{K_1}\mathbb{P}_{\mathcal{C}'}u:=u^*$, denote its contribution to $h_{k,\xi}$ by $h_{k,\xi}^*$, then we know by (\ref{xsb02}) that \[\|h_{k,\xi}^*\|_{X^{1,3/4}}\lesssim \|\mathcal{D}(e^{i(k\cdot x+\xi t)}u^*)\|_{X^{1,-1/4}},\] which is then bounded by\[\int_{\mathbb{R}\times\mathbb{T}^3}Ke^{i(k\cdot x+\xi t)}\mathcal{D}\overline{v}\cdot u^*\,\mathrm{d}x\mathrm{d}t\] by duality, where $\|v\|_{X^{0,1/4}}\lesssim 1$. By translation we can set $m=0$; by Plancherel, the above can be written as\[K\sum_{|n|\sim K}\int_{\mathbb{R}}\widehat{u^*}(n,\zeta)\overline{\widehat{\mathcal{D}v}(n+k,\zeta+\xi)}\,\mathrm{d}\zeta.\] Note that $|k|\lesssim K^{\alpha_1}$ and $|\xi|\lesssim K^{30\alpha_1}$, so we know that\[\max(|\zeta+Q(n)|,|\zeta+\xi+Q(n+k)|)\geq|Q(n+k)-Q(n)|-O(1)K^{30\alpha_1}\gg K^{60\alpha_1}\] since $n\not\in\mathcal{C}$. Since $\|\mathcal{D}u^*\|_{X^{0,1/6}}\lesssim K^{-1}$ and thus $\|u^*\|_{X^{0,1/6}}\lesssim K^{-1}\min(1,(K/N)^{-s+1})$, and $\|\mathcal{D}v\|_{X^{0,1/4}}\lesssim \max(1,(K/N)^{s-1})$, the above can be bounded by extracting a factor of either $|\zeta+Q(n)|$ or $|\zeta+\xi+Q(n+k)|$ and estimating both factors in $\ell_n^2L_\zeta^2$. This gives that $\|h_{k,\xi}^*\|_{X^{1,3/4}}\lesssim \eta K^{-10\alpha_1}$, and thus \[\eta^4K^{7\alpha_1}\sup_{5\leq q\leq 12}K^{\frac{5}{q}-\frac{1}{2}}\|h_{k,\xi}^*\|_{L_{t,x}^{q}[m,m+1]}+\|e^{i(t-m-1)\Delta}h_{k,\xi}^*(m+1)\|_{S_{J,K}^{7/2}}\lesssim \eta^5 K^{-3\alpha_1}.\] Summing up, we get that \[\|v_{m,K_1,\cdots,K_5}'\|_{S_{K,J}^{7/2}}+\|v_{m,K_1,\cdots,K_5}''\|_{S_{K,J}^{7/2}}\lesssim\eta^5\max(K^{-3\alpha_1,K^{-1/6+15\alpha_1}})\lesssim \eta^5K^{-1/100},\]since we have chosen $\alpha_1=1/100$.

Now, combining the results from (1), (2) and (3), and summing up in $m$ where $0\leq m\leq T'$, we get that\[A_K\lesssim a_K+\eta^5K^{-1/2000}+\eta\sum_M\min\bigg(\frac{K}{M},\frac{M}{K}\bigg)^{1/256}\sum_{0\leq m\leq T'}\sup_{5\leq q\leq 12}M^{\frac{5}{q}-\frac{1}{2}}\|P_M\mathcal{D}u\|_{L_{t,x}^q[m,m+1]\cap J}^{7/2}.\] Since $|J|\leq K^\gamma$, the summation in $m$ with any fixed $M$ can be divided into $O(K/M)^\gamma$ terms if $M\lesssim K$, and $O(1)$ terms if $M\gtrsim K$, such that each term is bounded by $\|P_M\mathcal{D}u\|_{S_{M,J'}^{7/2}}^{7/2}\lesssim A_M^{7/2}$ for some $J'\subset[0,T]$ such that $|J'|\leq M^\gamma$. Therefore we get \[A_K\lesssim a_K+\eta^5K^{-1/2000}+\sum_M\min\bigg(\frac{K}{M},\frac{M}{K}\bigg)^{1/256}\max\bigg(\frac{K}{M},\frac{M}{K}\bigg)^{\gamma}A_M^{7/2}.\] Since $\gamma=1/300$ and $A_K\lesssim 1$, this implies (\ref{double}).

Now that (\ref{double}) is proved, let $c_K=a_K+\eta K^{-1/2000}$, we can use Lemma \ref{aux} to conclude that\[A_K\lesssim\sup_{M}\min\bigg(\frac{M}{K},\frac{K}{M}\bigg)^{1/6000}(a_M+\eta M^{-1/2000}).\] Since $a_M\lesssim\eta$, this immediately implies $A_K\lesssim\eta$; moreover for $K\geq N$ we have\[A_K\lesssim \eta K^{-1/6000}+\sum_M \min\bigg(\frac{M}{K},\frac{K}{M}\bigg)^{1/6000}(M\|P_M\mathcal{D}f\|_{L^2})+\eta(N/K)^{1/6000}.\] Since \[\sum_M (M\|P_M\mathcal{D}f\|_{L^2})^2\sim\|\mathcal{D}f\|_{H^1}^2\lesssim\eta^2,\] by Schur's inequality it is easily seen that \[\sum_{K\geq N}A_K^2\lesssim\eta^2.\] Thus (\ref{nonlinear2}) is proved.
\end{proof}
\begin{rem}\label{transl} The reason we can assume $J=[0,T']$ when $K<N$ is because $a_K=\eta$. When we translate $J$ in time we have to replace $f$ by $u(t)$ where $t$ is the left endpoint of $J$. Since one still has $K\|P_Ku(t)\|_{L^2}\lesssim\eta$, the above proof will carry over to this case.
\end{rem}
With Corollary \ref{coro} and Proposition \ref{nonlctrl}, we can now finish the proof of Theorem \ref{main}.
\begin{proof}[Proof of Theorem \ref{main}] Let $\|u(0)\|_{H^s}=A$. Fix a large enough constant $D$ not depending on $\eta$. In this proof any implicit constant $C$ appearing in $\lesssim$ will be $\ll D$. Let $u$ be a solution to (\ref{nls}) with energy $E[u]\lesssim\eta^2$, as described in Proposition \ref{locale}. Choose $N$ such that $\|u(0)\|_{H^s}\sim\eta N^{s-1}$, then with $\mathcal{D}=\mathcal{D}_N$ one has that $\|\mathcal{D}u(0)\|_{H^1}\lesssim\eta$. By Proposition \ref{locale} and Strichartz, we see that for $T=1$, \begin{equation}\label{finalboot}\sup_{0\leq t\leq T}\|\mathcal{D}u(t)\|_{H^1}\leq D\eta,\quad A_K(T)\leq 1,\end{equation} with $A_K(T)$ defined in Proposition \ref{nonlctrl}. Suppose (\ref{finalboot}) for some $T\leq N^\gamma$, then by Corollary \ref{coro} and Proposition \ref{nonlctrl}, for $t\in[0,T]$ we have $A_K(T)\lesssim\eta\ll 1$. Moreover we have\[\begin{split}E[\mathcal{D}u(t)]-E[\mathcal{D}u(0)]&\lesssim_D\eta^2\sum_{K}\min\bigg(1,\frac{K}{N}\bigg)^{1/6}\sup_{|J|\leq K^{\gamma},J\subset[0,T]}\|P_K\mathcal{D}u\|_{S_{K,J}^{4}}^4\\
&\lesssim_D \eta^2\sum_{K}\min\bigg(1,\frac{K}{N}\bigg)^{1/6}\sup_{|J|\leq K^{\gamma},J\subset[0,T]}\|P_K\mathcal{D}u\|_{S_{K,J}^{7/2}}^4\\
&\lesssim_D \eta^2\sum_{K}\min\bigg(1,\frac{K}{N}\bigg)^{1/6}A_K^4\\
&\lesssim_D \eta^4\sum_{K\lesssim N}\bigg(\frac{K}{N}\bigg)^{1/6}A_K^2+\eta^4\sum_{K\gtrsim N}A_K^{2}\lesssim_D\eta^6.\end{split}\] Since $\eta$ is sufficiently small, this implies that $E[\mathcal{D}u(t)]\leq E[\mathcal{D}u(0)]+O_D(1)\eta^6\leq C\eta^2$, which gives \[\sup_{0\leq t\leq T} \|\mathcal{D}u(t)\|_{H^1}\ll D\eta.\] By bootstrap arguments, this implies that (\ref{finalboot}) remains true up to $T=N^{\gamma}$, which implies $\|u(t)\|_{H^s}\lesssim\eta N^{s-1}$ for $0\leq t\leq T$.

Using time translation and rescaling $N$ by a factor depending on $\eta$, we get the following result with some absolute constant $E$ (which could depend on $\eta$):
\begin{equation}\label{lac}\text{If $\|u(t)\|_{H^s}\leq N^{s-1}$, then for $|t'-t|\leq N^{\gamma}/E$ we have $\|u(t')\|_{H^s}\leq  (EN)^{s-1}$.}\end{equation} Now, for any positive integer $m$ such that $E^{m(s-1)}\geq A$, choose the smallest time $t_m>0$ such that $\|u(t)\|_{H^s}\geq  E^{m(s-1)}$, then by (\ref{lac}) we have\[t_{m+1}-t_m\geq E^{m\gamma-1},\] so in particular $t_m\gtrsim E^{m\gamma}$. Therefore, for each $t>0$, if $\|u(t)\|_{H^s}\gg A$, choosing the biggest $m$ such that $\|u(t)\|_{H^s}\geq  E^{m(s-1)}$, we get that $t\geq t_m\gtrsim E^{m\gamma}$, thus \[\|u(t)\|_{H^s}\leq E^{(m+1)(s-1)}\lesssim E^{m(s-1)}\lesssim t^{\frac{s-1}{\gamma}}.\] The negative times are proved in the same way. Since $\gamma=1/300$ this completes the proof of Theorem \ref{main}.
\end{proof}

\end{document}